\documentclass[12pt,reqno]{amsart}
\usepackage{graphicx,color}
\usepackage{amssymb,amscd, amsthm, latexsym, mathrsfs, epsfig}
\usepackage[all]{xy}
\usepackage{todonotes}

\newcommand{\g}{\gamma}
\newcommand{\gd}{\gamma^{\vee}}

\newcommand\PP{\mathbb P}

\newcommand\C{\mathbb C}

\newcommand\Q{\mathbb Q}
\newcommand\R{\mathbb R}
\newcommand\Z{\mathbb Z}

\renewcommand{\H}{\mathbb{H}}

\newcommand{\U}{\mathcal{U}}
\newcommand{\V}{\mathcal{V}}

\newcommand\Aut{\operatorname{Aut}}

\newcommand\dt{\operatorname{DT}}

\newcommand\bra{\langle}
\newcommand\ket{\rangle}
\newcommand{\pow}[1]{[\![ {#1} ]\!]}

\newcommand{\del}{\partial}

\newcommand{\Hom}{\operatorname{Hom}}

\makeatletter \@addtoreset{equation}{section} \makeatother

\newtheorem{thm}{Theorem}
\newtheorem{prop}[thm]{Proposition}
\newtheorem{lem}[thm]{Lemma}
\newtheorem{cor}[thm]{Corollary}

\theoremstyle{definition}
\newtheorem{definition}[thm]{Definition}
\newtheorem{exm}[thm]{Example}
\newtheorem{rmk}[thm]{Remark}

\newcommand\narrowdots{\hbox to 1em{.\hss.\hss.}}

\title[Variations of BPS structure]{Variations of BPS structure\\ and a large rank limit}
\date{14 May 2017}
\author{Jacopo Scalise and Jacopo Stoppa}
\email{jscalise@sissa.it}
\email{jstoppa@sissa.it}
\address{SISSA, via Bonomea 265, 34136 Trieste, Italy}
\begin{document}

\begin{abstract} We study a class of flat bundles, of finite rank $N$, which arise naturally from the Donaldson-Thomas theory of a Calabi-Yau threefold $X$ via the notion of a variation of BPS structure. We prove that in a large $N$ limit their flat sections converge to the solutions to certain infinite dimensional Riemann-Hilbert problems recently found by Bridgeland. In particular this implies an expression for the positive degree, genus $0$ Gopakumar-Vafa contribution to the Gromov-Witten partition function of $X$ in terms of solutions to confluent hypergeometric differential equations. 
\end{abstract}

\maketitle

\section{Introduction and main results}
In this Introduction we describe the circle of ideas and main results of this paper. All definitions and proofs are given in the following sections.  

Let $X$ be a complex projective Calabi-Yau threefold. Write $\Gamma$ for its numerical Grothendieck group endowed with the skew-symmetric bilinear Euler form $\bra - , - \ket$. Some of the aims of (generalised, unrefined) Donaldson-Thomas theory (see \cite{js, ks}) are  
\begin{enumerate}
\item to define deformation invariants $\dt(\alpha, Z) \in \Q$, virtually enumerating objects in $D^b(X)$ which have prescribed class $\alpha \in \Gamma$ and which are semistable with respect to a numerical Bridgeland stability condition, locally described by an element $Z \in \Hom(\Gamma, \C)$;
\item to define underlying (``BPS") invariants $\Omega(\alpha, Z) \in \Q$ via a known, universal multi-cover formula, and to prove that in fact they take values in $\Z$ (at least for sufficiently general $Z$);
\item to prove that the variation of $\dt(\alpha, Z)$ (equivalently $\Omega(\alpha, Z)$) when we deform the stability condition $Z$ is given by a known, universal expression, the JS/KS wall-crossing formula (due to Joyce-Song and Kontsevich-Soibelman).  
\end{enumerate}

Thanks to the work of several authors these aims have now been achieved in some special but highly nontrivial examples (see in particular \cite{macri, toda}). A much simpler case is discussed at the end of this Introduction. 

This general theory leads to formulate the abstract notions of a BPS structure $(\Gamma, Z, \Omega)$ on a lattice $\Gamma$ with a form $\bra - , - \ket$, and of its variation, which simply describe the outcome of (1)-(2) above for a fixed $Z$, respectively (3) above when varying $Z$. So $Z$ is an element of $\Hom(\Gamma, \C)$ and $\Omega$ a map of sets $\Gamma \to \Q$ (or $\Gamma \to \Z$ in the integral case), satisfying certain constraints, including the JS/KS formula when $Z$ varies. The function $\dt$ is defined from $\Omega$ by inverting the multi-cover formula. 

This idea is due to Kontsevich and Soibelman (\cite{ks} Section 2, \cite{ksWall} Section 2). It is somewhat analogous to introducing the abstract notion of a (variation of) Hodge structure starting from the case of (a family of) Kaehler manifolds. In this analogy the JS/KS formula may be compared to Griffiths transversality. The terminology used in the present paper was introduced by Bridgeland in \cite{bridRH} to cover special cases of Kontsevich and Soibelman's more general notions of stability data and wall-crossing structures. Important motivation for this abstract approach comes from the fact that variations of BPS structure appear naturally in other contexts, notably in symplectic geometry (see e.g. \cite{chan, ksWall, lin}) and in the Gross-Siebert program for mirror symmetry (via scattering diagrams, see e.g. \cite{bridScatter, gross, gps}).

One of the main aims of the present paper is to show how some very special but interesting variations of BPS structure (which correspond roughly to the case of torsion coherent sheaves on $X$ supported in dimension at most $1$) can be described effectively in terms of classical objects, namely linear complex differential equations of hypergeometric type. At the same time we relate this description to recent work of Bridgeland \cite{bridRH}. As an application we find an expression for the positive degree, genus $0$ Gopakumar-Vafa contribution to the Gromov-Witten partition function of a Calabi-Yau threefold $X$ in terms of solutions to confluent hypergeometric differential equations. 

We follow two closely related approaches, based respectively on Riemann-Hilbert factorisation problems (RH problems) and on flat bundles (of Frobenius type). In our loose analogy with variations of Hodge structure the latter correspond to the Gauss-Manin connection, the former to the inverse problem of reconstructing the Gauss-Manin from its monodromy. 

RH problems are a special type of boundary value problems for holomorphic functions, much studied in complex analysis and mathematical physics (see e.g. \cite{painleveBook}). 
A BPS structure $(\Gamma, Z, \Omega)$ induces in a very natural way various RH problems for maps from $\C^*$ to an affine algebraic torus $\mathbb{T}$, given by characters of $\Gamma$ twisted by the form $\bra -, -\ket$. Unlike the classical case the corresponding structure group is always infinite dimensional, and for the purposes of the present paper it is a subgroup of $\operatorname{Bir}(\mathbb{T})$. This idea is due to Gaiotto, Moore and Neitzke (see \cite{gmn}) and was studied e.g. in \cite{bridRH, fgs, iwaki}. Let us recall a recent result in this connection, concerning the case of (finite) uncoupled BPS structures. These are the simplest objects in the theory, defined by the condition that the Euler pairing vanishes when restricted to the locus $\Omega \neq 0$, i.e. to active classes (which are finitely many, in the finite case). In particular we will see that the function $\Omega$ is constant along a variation of uncoupled BPS structure. Geometrically these structures correspond to the case of torsion coherent sheaves on $X$ supported in dimension at most $1$, as discussed at the end of this Introduction.
Define a multi-valued meromorphic function on $\C^*$ by
\begin{equation*}
\Lambda(w) = \frac{e^w \Gamma(w)}{\sqrt{2\pi} w^{w-\frac{1}{2}}},
\end{equation*}
where $\Gamma(w)$ is the classical gamma function (see e.g. \cite{erdelyi} Chapter I). If $\ell \subset \C^*$ is a ray emanating from $0 \in \C$ we introduce the half-plane
\begin{equation*}
\H_{\ell} = \{z \in \C^* | z = u v \text{ with } u \in \ell \text{ and } \Im(v) > 0\} \subset \C^*.
\end{equation*}
\begin{thm}[Bridgeland \cite{bridRH} Theorem 5.3]\label{bridThmA} Let $(\Gamma, Z, \Omega)$ be a finite, integral, uncoupled BPS structure. Suppose $\xi \in \mathbb{T}$ is such that $\xi(\gamma) = 1$ when $\Omega(\gamma) \neq 0$. Then the infinite dimensional, birational RH problem of $(\Gamma, Z, \Omega)$ at $\xi$ has a unique solution $\Psi(t)$, with component along $\beta \in \Gamma$ given explicitly by the collection of functions 
\begin{equation*}
\Psi_{\H_{\ell}, \beta}(t) = \prod_{\g | \Omega(\g)\neq 0, Z(\g) \in \H_{\ell}} \Lambda\left(\frac{Z(\g)}{t}\right)^{\Omega(\g)\bra\beta, \g\ket},\, t \in \H_{\ell}
\end{equation*} 
for generic $\ell \subset \C^*$.
\end{thm}
We will relate this infinite dimensional result to large rank limits of classical, finite dimensional flat bundles (i.e. systems of linear complex ODEs).

A central notion for us is that of a Frobenius bundle, introduced by Hertling (following Dubrovin \cite{dubrovin}) in his study of geometric structures on unfolding spaces of singularities (see \cite{hert} Section 5.2). A Frobenius bundle $K$ is a holomorphic bundle over a complex manifold $M$ with additional data, including a flat connection $\nabla^r$, a Higgs field $C$ and a holomorphic quadratic form $g$ (the ``metric"). Under some assumptions Barbieri and the second author (see \cite{AnnaJ}) show that there is a natural correspondence between variations of BPS structure and Frobenius bundles of a special form. The main ingredient is a holomorphic generating function $f(Z)$ for the invariants $\dt(\alpha, Z)$ introduced by Joyce (see \cite{joy}).
\begin{prop}[\cite{AnnaJ} Proposition 3.17]\label{AnnaJProp} There is a natural correspondence between 
\begin{enumerate}
\item framed variations of BPS structure $(\Gamma, Z, \Omega)$ over a complex manifold $M$, such that $Z$ takes values in the complement of a line through $0 \in \C$, endowed with the choice of a basis for $\Gamma$;
\item Frobenius bundle structures $K$ on the trivial bundle over $M$ with fibre the group algebra $\C[\Gamma]$, with values in formal power series, such that the Higgs field $C$ equals $-dZ$ and the flat connection $\nabla^r$ is given by the adjoint action of $f(Z)$. 
\end{enumerate}  
If $(\Gamma, Z, \Omega)$ is uncoupled this holds without the extra assumption on $Z$. 
\end{prop}
Note that the bundle $K$ is infinite dimensional, generated by the global sections $x_{\alpha}$, $\alpha \in \Gamma$ corresponding to the generators of the group algebra. For all finite subsets $\Delta = \{\alpha_i\} \subset \Gamma$ there is a finite dimensional subbundle $K_{\Delta}\subset K$ spanned by $\{x_{\alpha_i}\}$, and the metric $g$ gives a canonical projection $K\to K_{\Delta}$. Our first result in this paper characterises uncoupled variations of BPS structure in terms of these finite dimensional subbundles.
\begin{thm}\label{thmA} Let $(\Gamma, Z, \Omega)$ be a framed variation of BPS structure over a complex manifold as in Proposition \ref{AnnaJProp}, $K$ the corresponding Frobenius bundle. The following are equivalent.
\begin{enumerate} 
\item The BPS structures in $(\Gamma, Z, \Omega)$ are uncoupled. 
\item For all $\Delta$ the canonical projection $K \to K_{\Delta}$ induces a Frobenius bundle structure on the finite dimensional subbundle $K_{\Delta} \subset K$.
\end{enumerate}
\end{thm}
\begin{rmk} We will see that in the uncoupled case the Frobenius bundles $K$,$K_{\Delta}$ fit in a $1$-parameter family $K_{\hbar}, K_{\Delta, \hbar}$ induced by rescaling the form 
\begin{equation}\label{hbarIntro}
\bra - , - \ket \mapsto i\hbar \bra - , - \ket
\end{equation}
for $\hbar \in \R_{> 0}$. This is a special case of a more general construction, which extends to the coupled case, see Remark \ref{deformation}. This deformation is natural from the point of view of refined Donaldson-Thomas theory, see Remark \ref{hbarBody}. 
\end{rmk}
Fix an uncoupled variation of BPS structure $(\Gamma, Z, \Omega)$ as above. The simplest nontrivial example of a Frobenius subbundle $K_{\Delta} \subset K$ has rank $2$ and is obtained by choosing $\Delta = \{m\g + m\beta, m\beta\}$ where $\gamma$ is an active class, $\bra \g, \beta\ket \neq 0$ and $m > 0$. We take into account the extra parameter $\hbar$ of the rescaling \eqref{hbarIntro} and call this Frobenius bundle $K_{\Delta,\hbar}$ the simple oscillator spanned by $\g$, $\beta$ with frequency $m$. We will see that this Frobenius bundle is determined by classical objects, namely $GL(2,\C)$ fundamental solutions $Y^{(m)}_{\hbar}(t)$ to the system of complex linear differential equations 
\begin{equation}\label{ODE}
\frac{\del}{\del t}Y^{(m)}_{\hbar}(t) = (-t^{-2}U^{(m)} + t^{-1}V^{(m)}_{\hbar})Y^{(m)}_{\hbar}(t)
\end{equation}
where
\begin{align*}
U^{(m)} &= \left(\begin{matrix}
m Z(\g + \beta) & 0\\ 0 & m Z(\beta) 
\end{matrix}\right),\\ V^{(m)}_{\hbar} &= 
\frac{\bra \g, \beta\ket \hbar}{2\pi}\Omega(\g)\left(\begin{matrix}
0 & (-1)^{m\bra\g,\beta\ket}\\ -(-1)^{m\bra\g,\beta\ket} & 0 
\end{matrix}\right).
\end{align*}
Turning the system into a single ODE in a standard way shows that $K_{\Delta,\hbar}$ is given by fundamental solutions to the confluent hypergeometric differential equation 
\begin{equation}\label{confluentODE}
u''(z) + \left(\frac{1}{z} - z_1 - z_2\right) u'(z) + \left(\frac{\mu^2}{z^2} - \frac{z_1}{z} + z_1 z_2\right) u(z) = 0
\end{equation}
where $z = t^{-1}$, $z_1 = m Z(\g + \beta)$, $z_2 = m Z(\beta)$, $\mu = -(-1)^{m\bra\g,\beta\ket}\frac{\bra \g, \beta\ket \hbar}{2\pi}\Omega(\g)$.  
\begin{rmk} By a slight abuse of notation we also call the standard normalisation $\Psi^{(m)}_{\hbar}(t)$ of the $GL(2, \C)$ fundamental solution $Y^{(m)}_{\hbar}(t)$ (determined by $\Psi^{(m)}_{\hbar}(t) \to I$ for $t\to 0$) a simple oscillator. We will show that $\Psi^{(m)}_{\hbar}(t) = I + O(\hbar)$ and $\log \Psi^{(m)}_{\hbar}(t) \in M_2(\C)$ is off-diagonal modulo $\hbar^2$ for all $t$. 
\end{rmk}
In view of Theorem \ref{thmA} it seems natural to ask if the function $\Psi_{\H_{\ell},\beta_j}(t)$ of Theorem \ref{bridThmA} can be recovered in a large rank limit, i.e. as the limiting behaviour along an infinite increasing sequence of Frobenius subbundles $K_{\Delta} \subset K$. One of our main results confirms this expectation. 
\begin{thm}\label{thmA} Let $(\Gamma, Z, \Omega)$ be a framed variation of uncoupled, finite BPS structure. Fix a basis $\{\beta_j\}$ for $\Gamma$ and let $\{\gamma_i\}$ be any collection of active classes. Let $\hat{\xi}$  denote the vector $\left(\begin{matrix}1 \\ 1\end{matrix}\right)\in\C^2$ and $\Pi$ be the linear functional on $\C^2$ given by $\Pi(w_1, w_2) = w_1 + w_2$. 
\begin{enumerate}
\item For all $N > 0$, the Frobenius bundle $K$ of $(\Gamma, Z, \Omega)$ contains a canonical, finite dimensional Frobenius subbundle isomorphic to the direct sum of all the simple oscillators spanned by $\g_i$, $\beta_j$ with frequency $m = 1, \cdots, N$.
\item Suppose now $\{\gamma_i\}$ is a maximal set of active classes such that all the $Z(\g_i)$ lie in a half-plane $\H_{\ell}$. Let $\Psi^{(m),ij}_{\hbar}$ denote the simple oscillator spanned by $\g_i$, $\beta_j$ with frequency $m$. Then we have  
\begin{align*}
&\exp\left(\frac{1}{\hbar}\sum^{\infty}_{m = 1} \sum_{i | \bra\g_i, \beta_j\ket\neq0} \frac{(-1)^{m\bra\g_i,\beta_j\ket}}{m}\Pi\log \Psi^{(m),ij}_{\hbar}((2\pi)^{-1} \sqrt{-1} t) \hat{\xi} \right) \\
&= \prod_{i} \Lambda\left(\frac{Z(\g_i)}{t}\right)^{\Omega(\g)\bra\beta_j, \g_i\ket} + O(\hbar)
\end{align*}
for all $t\in \C^*$ such that $\Re(Z(\g_i)/t) > 0$ for all $i$. Integrality is not required. For finite, integral BPS structures the latter product equals the solution $\Psi_{\H_{\ell},\beta_j}(t)$ given by Theorem \ref{bridThmA}, i.e. the solution to the infinite dimensional, birational RH problem of $(\Gamma, Z, \Omega)$ at $\xi \in \mathbb{T}$ is the leading order term in the $\hbar \to 0$, $N \to \infty$ limit of a sum of simple oscillators, in the nonempty open sector of $\H_{\ell}$ where $\Re(Z(\g_i)/t) > 0$ for all $i$.
\end{enumerate}
\end{thm}
\begin{rmk}\label{explicitPsiIntro} Evaluating at $\hat{\xi}$ (more precisely at $\oplus_{i, m}\hat{\xi}$) is the finite dimensional analogue of evaluating at a special point $\xi \in \mathbb{T}$ as in Theorem \ref{bridThmA}. Similarly the linear functional $\oplus_{i,m} \frac{(-1)^{m\bra\g_i,\beta_j\ket}}{m}\Pi$ is the finite dimensional analogue of the torus character projecting along the $\beta_j$ component as in Theorem \ref{bridThmA}. In terms of matrix entries we have
\begin{equation*}
\Pi\log \Psi^{(m),ij}_{\hbar}((2\pi)^{-1} \sqrt{-1} t) \hat{\xi} = \Psi^{(m),ij}_{\hbar}((2\pi)^{-1} \sqrt{-1} t)_{(12)}
\end{equation*}
where for a matrix $A$ we write $A_{(kl)} = A_{kl} + A_{lk}$. We will see that in fact there is an explicit formula
\begin{align*}
\Pi\log \Psi^{(m), ij}_{\hbar}(t)\hat{\xi} &= -(-1)^{m\bra\g_i,\beta_j\ket} m \bra\g_i, \beta_j\ket\hbar\Omega(\g_i)\\&\frac{1}{\pi} \int^{\infty}_0 \arctan\left(\left(\frac{Z(\g_i)}{t}\right)^{-1} s\right)  e^{-m s} ds + O(\hbar^2).
\end{align*}
\end{rmk}

\begin{rmk} Theorem \ref{bridThmA} and the proof of Theorem \ref{thmA} are inspired by a calculation of Gaiotto (see \cite{gaiotto} Section 3.1). We note that the idea of looking at large rank, weak coupling limits of the form $\hbar \to 0$, $N \to \infty$ is familiar from the ``large $N$ limit" in the theory of matrix models, with the standard notation $g_s = 1/N$, $N \to \infty$ (see e.g. \cite{marinoBook} Chapter I Section 1.1). It seems interesting to ask if the higher order terms in the $\hbar$ expansion of Theorem \ref{thmA} (2) also have a natural interpretation. 
\end{rmk}

We consider now the case when $(\Gamma, Z, \Omega)$ is a miniversal variation of finite, integral BPS structure. This means that fixing a basis $\{\beta_j\}$ one can use the central charges $Z(\beta_j)$ as local coordinates on the base. If $v^j(t, Z)$ is a vector function of $t, Z$ with vector index $j$, we follow \cite{bridRH} Section 3.4 and define a tau function $\tau_v$ for $v$ as a solution to
\begin{equation}\label{tauEquIntro} 
\frac{\del}{\del t}\log v^j = \sum_p \bra \beta_j, \beta_p\ket \frac{\del}{\del Z(\beta_p)}\log\tau_v,
\end{equation}
for all $j$, which is invariant under a common rescaling of $t$ and all $Z(\beta_j)$. Define a multi-valued meromorphic function on $\C^*$ by
\begin{equation*}
\Upsilon(w) = \frac{-\zeta'(-1)e^{\frac{3}{4}w^2} G(w+1)}{(2\pi)^{w/2} w^{w^2/2}},
\end{equation*} 
where $G(w)$ is the Barnes $G$-function (see \cite{whit} p. 264).
\begin{thm}[Bridgeland Theorem 3.4]\label{bridThmB} Let $(\Gamma, Z, \Omega)$ be a miniversal variation of finite, integral, uncoupled BPS structure. Then the vector function $\Psi_{\H_{\ell}, \beta_j}$ (vector index $j$) admits the tau function
\begin{equation}\label{bridTau}
\tau_{\ell}(t, Z) = \prod_{\g | \Omega(\g)\neq 0, Z(\g) \in \H_{\ell}} \Upsilon\left(\frac{Z(\g)}{t}\right)^{\Omega(\g)},\, t \in \H_{\ell}.
\end{equation}
\end{thm}
The tau function $\tau_{\ell}(t, Z)$  plays an important role because it can be related more directly to Gromov-Witten partition functions, as we explain below. We can prove an analogue of Theorem \ref{thmA} for tau functions. Write $\{\g_i\}$ for the active classes as above. Introduce the scalar functions
\begin{align}\label{explicitTauIntro}
\log \tau^{(m),i}_{\hbar}(\sqrt{-1}t) =      \frac{\Omega(\g_i)}{2\pi} \hbar  \int^{\infty}_0  s\log\left(s^2 + \left(\frac{Z(\g_i)}{t}\right)^2\right)  e^{-m s} ds
\end{align}
(compare to the explicit formula in Remark \ref{explicitPsiIntro}). 
\begin{thm}\label{thmB} Let $(\Gamma, Z, \Omega)$ be a miniversal variation of finite, uncoupled BPS structure. Let notation and assumptions be as in Theorem \ref{thmA}. 
\begin{enumerate} 
\item The vector function (vector index $j$)
\begin{equation*}
\exp\left(\frac{1}{\hbar}\sum^{\infty}_{m = 1} \sum_{i | \bra\g_i, \beta_j\ket\neq0} \frac{(-1)^{m\bra\g_i, \beta_j\ket}}{m}\Pi\log \Psi^{(m),ij}_{\hbar}((2\pi)^{-1} \sqrt{-1} t)\hat{\xi} \right)
\end{equation*}
admits the tau function
\begin{equation*} 
\exp\left(\frac{1}{\hbar}\sum^{\infty}_{m = 1} \sum_{i | \Omega(\g_i) \neq 0} \log \tau^{(m), i}_{\hbar}((2\pi)^{-1} \sqrt{-1} t) \right)
\end{equation*}
modulo $\hbar$, i.e. this solves \eqref{tauEquIntro} up to $O(\hbar)$. 
\item In the integral case the latter function equals the tau function $\tau_{\ell}(t, Z)$ given by \eqref{bridTau}. By Theorem \ref{thmA} (2) this implies Theorem \ref{bridThmB}, i.e. the tau function $\tau_{\ell}(t, Z)$ of the infinite dimensional, birational RH problem of $(\Gamma, Z, \Omega)$ is the tau function for the leading order term in the $\hbar \to 0$, $N \to \infty$ expansion of a sum of simple oscillators (at least in a nonempty, open sector).
\end{enumerate}
\end{thm}
Let us return to the geometric case of a Calabi-Yau threefold $X$. Theorem \ref{thmB} can be used in conjunction with results from \cite{bridRH} to show that a certain (Gopakumar-Vafa) contribution to the Gromov-Witten partition function of $X$ can be expressed in terms of solutions to the confluent hypergeometric equation \eqref{confluentODE}, i.e. in terms of a sum of simple oscillators. 

To explain this we recall that Bridgeland (\cite{bridRH} Section 6) constructs a miniversal variation of uncoupled BPS structure where $\Gamma = H_{2*}(X, \Z)$ (modulo torsion), $\bra - , - \ket$ is the intersection pairing, and $\Omega(\alpha)$ vanishes except when $\alpha = (n, \beta, 0, 0)$, when it is the BPS invariant enumerating coherent sheaves on $X$ supported in dimension $\leq 1$ and with Chern character dual to $\alpha$ (see \cite{js} Section 6). Central charges of active classes are specified by $Z(n, \beta, 0, 0) = \int_{\beta}\omega_{\C} - n$, $\omega_{\C}$ denoting a complexified K\"ahler class. Note that these BPS structures are not finite. Their formal tau function is given by the right hand side of \eqref{bridTau}, regarded as a formal infinite product. 
\begin{prop}[Bridgeland-Iwaki \cite{bridRH} Section 6.3] Consider the positive degree, genus $0$ Gopakumar-Vafa contribution to the Gromov-Witten partition function of $X$, given explicitly by
\begin{align}\label{GV}
\nonumber &\chi(X)\sum_{g \geq 2} \frac{(-1)^{g-1} B_{2g} B_{2g-2}}{4g(2g-2)(2g-2)!}\lambda^{2g-2}\\
&+ \sum_{g\geq 2} \sum_{\beta \in H_2(X, \Z)} \operatorname{GV}(0, \beta) \frac{(-1)^{g-1}B_{2g}}{2g(2g-2)!}\operatorname{Li}_{3-2g}(x^{\beta})\lambda^{2g-2}.
\end{align} 
Assuming the conjectural relation $\Omega(n, \beta, 0, 0) = \operatorname{GV}(0, \beta)$ for all $n$, with $\beta$ a positive curve class (see \cite{js} Conjecture 6.20), the change of variables 
\begin{equation}\label{GWDT}
\lambda = 2\pi t, \, x^{\beta} = \exp (2\pi i v_{\beta}), \, v_{\beta} = \int_{\beta} \omega_{\C}
\end{equation}
gives the logarithm of the formal tau function for sheaves on $X$ supported in dimension $\leq 1$, i.e. the logarithm of the right hand side of \eqref{bridTau} regarded as a formal infinite product.
\end{prop}   
The following result thus follows immediately from Theorem \ref{thmB}. 
\begin{cor} After the change of variables \eqref{GWDT}, the positive degree, genus $0$ Gopakumar-Vafa contribution to the Gromov-Witten partition function of $X$ \eqref{GV} can be written as a sum of simple oscillator tau functions
\begin{equation*}
\frac{1}{\hbar}\sum^{\infty}_{m = 1}  \sum_{\beta, n} \log \tau^{(m), (n, \beta, 0, 0)}_{\hbar}((2\pi)^{-1} \sqrt{-1} t)
\end{equation*}
regarded as a formal power series in $t$, $v_{\beta}$, where $\log \tau^{(m), (n, \beta, 0, 0)}_{\hbar}$ is given by setting $\g_i = (n, \beta, 0, 0)$ in the right hand side of \eqref{explicitTauIntro}.
\end{cor}
\begin{rmk} Bridgeland \cite{bridRH2} has shown how to extend Theorems \ref{bridThmA}, \ref{bridThmB} to the variation of BPS structure of sheaves on $X$ with dimension $\leq 1$ when $X$ is the resolved conifold. The corresponding tau function turns out to be another classical special (double sine) function. We expect that this function can be recovered from sums of simple oscillators as in Theorem \ref{thmB}.  
\end{rmk}
\noindent\textbf{Plan of the paper.} Section \ref{backSec} contains the required background on BPS structures, their variations, and the associated Frobenius bundles. Sections \ref{A1Sec1}, \ref{A1Sec2} and \ref{A1Sec3} discuss and prove Theorem \ref{thmA} for the special case of rank $2$ BPS structures, i.e. when $\operatorname{rk}(\Gamma) = 2$. Section \ref{UncoupledSec} completes the proof for arbitrary rank of $\Gamma$. Given the results of the previous sections this is mostly a matter of notation. Section \ref{tauSec} proves Theorem \ref{thmB}.\\

\noindent\textbf{Acknowledgements.} We are very grateful to Anna Barbieri, Tom Bridgeland, Giordano Cotti and especially to Davide Guzzetti for helpful discussions and comments on our work. The research leading to these results has received funding from the European Research Council under the European Union's Seventh Framework Programme (FP7/2007-2013) / ERC Grant agreement no. 307119. 
\section{BPS structures and Frobenius bundles}\label{backSec}
In this Section we introduce BPS structures, their variations, and the corresponding Frobenius bundles. Since there are already many references for this material we are very brief.
\begin{rmk} Definitions \ref{BPSauto}, \ref{definitionRH} and the wall-crossing identity \eqref{wallcross} below are only given for the sake of motivation, in incomplete form. They are never used in the present paper. However we will point out the main difficulties involved and give references which contain a fully rigorous treatment. 
\end{rmk}
\begin{definition}[\cite{bridRH} Section 2.1, \cite{ks} Section 2] A BPS structure comprises a finite rank lattice $\Gamma$ (charge lattice), endowed with a skew-symmetric integral bilinear form $\bra - , - \ket$ (intersection form), an element $Z \in \Hom(\Gamma, \C)$ (central charge) and a map of sets $\Omega\!: \Gamma \to \Q$ (BPS spectrum), with constraints given by $\Omega(\alpha) = \Omega(-\alpha)$ (symmetry) and the property that there is a fixed $C > 0$ such that $\Omega(\gamma) \neq 0$ implies 
\begin{equation*}
|Z(\gamma)| > C|| \gamma ||
\end{equation*}
for some fixed choice of norm on $\Gamma \otimes\R$ (support property). The rank of a BPS structure is the rank of $\Gamma$. We say that a BPS structure is integral if $\Omega$ takes values in $\Z$.
\end{definition}
Note that the required symmetry models the shift functor $[1]$ acting on $D^b(X)$.
\begin{definition}[\cite{bridRH} Section 2.2, \cite{ks} Section 2.5] Let $(\Gamma, Z, \Omega)$ be a BPS structure. The corresponding DT spectrum is the map of sets $\dt\!: \Gamma \to \Q$ defined by
\begin{equation*}
\dt(\alpha) = \sum_{k > 0 | k^{-1}\alpha \in \Gamma} \frac{\Omega(\alpha/k)}{k^2}.
\end{equation*}
The maps $\Omega$, $\dt$ are equivalent data (by a standard inversion formula).
\end{definition} 
\begin{definition}[\cite{bridRH} Section 2.1] An element $\gamma \in \Gamma$ is called an active class if $\Omega(\gamma) \neq 0$. An active ray $\ell \subset \C^*$ is a ray of the form $\R_{>0}Z(\gamma)$ where $\gamma$ is an active class. We say $\ell$ is generic if it is not active. A BPS structure is finite if there are finitely many active classes.
\end{definition}
The following definition is central to this paper.
\begin{definition}[\cite{bridRH} Definition 2.3, \cite{gmn} Section 4] We say that a BPS structure $(\Gamma, Z, \Omega)$ is uncoupled if we have $\bra \g_i, \g_j\ket = 0$ for all active classes $\g_i$.
\end{definition}
To formulate the correct notion of a variation we need some further ingredients.
\begin{definition} In this paper we always denote by $\C[\Gamma]$ the group-algebra of $\Gamma$ endowed with the twist of the usual associative, commutative product by the form $\bra -, -\ket$,
\begin{equation*}
x_{\alpha} x_{\beta} = (-1)^{\bra \alpha, \beta\ket} x_{\alpha + \beta}.
\end{equation*}
The torus of twisted characters is the affine algebraic torus
\begin{equation*}
\mathbb{T} = \operatorname{Spec}\C[\Gamma]
\end{equation*}
We write $\mathbb{T}_+$ for the usual affine algebraic torus $\operatorname{Spec}\C[\Gamma]_*$, where $\C[\Gamma]_*$ denotes the usual group-algebra with untwisted commutative product. Then $\mathbb{T}$ is a torsor for $\mathbb{T}_+$ (see \cite{bridRH} Section 2.4, \cite{ks} Section 2.5).
\end{definition}
Note that one can think of $x_{\alpha} \in \C[\Gamma]$ as a map $\mathbb{T} \to \C^*$ (a twisted character), and similarly of $y_{\alpha} \in \C[\Gamma]_*$ as a usual character $\mathbb{T}_+ \to \C^*$.
\begin{lem}[\cite{bridRH} Section 2.4, \cite{ks} Section 2.5] The pairing 
\begin{equation*}
[x_{\alpha}, x_{\beta}] = (-1)^{\bra \alpha, \beta\ket} \bra \alpha, \beta \ket x_{\alpha + \beta}
\end{equation*}
defines a Poisson bracket on the commutative algebra $\C[\Gamma]$. 
\end{lem}
\begin{proof} This is a straightforward computation.
\end{proof}
\begin{definition}[\cite{bridRH} Section 2.5, \cite{ks} Section 2.5]\label{BPSauto} Given a ray $\ell$ we define 
\begin{equation*}
\dt(\ell) = \sum_{\g \in \Gamma | Z(\g ) \in \ell } \dt(\g) x_{\g}.
\end{equation*}
The BPS automorphism of an active ray $\ell$ is
\begin{equation*}
\exp([ \dt(\ell), - ]) \in \Aut(\C[\Gamma]).
\end{equation*}
\end{definition}
\begin{rmk} The sum defining $\dt(\ell)$ is either empty or infinite, and the vector field $[ \dt(\ell), - ]$ may be ill-defined. It turns out that one can always make sense of $\exp([ \dt(\ell), - ])$ as a formal automorphism, and when the BPS structure in finite and integral $\exp([ \dt(\ell), - ])$ is in fact an element of $\operatorname{Bir}(\mathbb{T})$, the group of birational automorphism of $\mathbb{T}$ (see \cite{bridRH} Section 2.7, \cite{ks} Section 2.5).
\end{rmk}
\begin{definition}[\cite{bridRH} Section 3.3, \cite{ks} Section 2.3] A variation of BPS structure is a family of BPS structures $(\Gamma_p, \Omega_p, Z_p)$ as above, parame\-trised by points $p$ of a complex manifold $M$, where the $\Gamma_p$ fit together in a local system, the $Z_p$ are holomorphic sections of $\Hom(\Gamma_p, \C)$ (the central charges), and the $\Omega(\alpha_p, Z_p)$ satisfy the JS/KS wall-crossing formula. This means that the product 
\begin{equation}\label{wallcross}
\prod_{\ell \subset V} S_p(\ell) \in \Aut(\mathbb{T}_p)
\end{equation}
is locally constant, where $\mathbb{T}_p$ is the local system of algebraic affine tori $\operatorname{Spec}(\C[\Gamma_p])$, $V \subset \C^*$ is the interior of a convex sector, and $\prod_{\ell \subset V}$ is computed writing the ensuing automorphisms from left to right according to the clockwise order of $\ell$. A variation is called framed if the local system $\Gamma_p$ is trivial. A framed variation is called miniversal if fixing a basis $\beta_j$ of $\Gamma$ induces local coordinates $Z(\beta_j)$ on $M$.
\end{definition} 
\begin{rmk} In general one regards \eqref{wallcross} as a formal automorphism and only imposes local constancy modulo a sequence of powers of a maximal ideal (see \cite{bridRH} Appendix A, \cite{ks} Section 2). When the BPS structures are finite and integral this is not necessary and one simply requires that \eqref{wallcross} is a locally constant section of $\operatorname{Bir}(\mathbb{T}_p)$. When the BPS structures are uncoupled the condition that \eqref{wallcross} is locally constant always holds automatically, since the $S_p(\ell)$ commute (this is clear from Definition \ref{BPSauto}).
\end{rmk}
Riemann-Hilbert problems are a classical topic in complex analysis and mathematical physics.
\begin{definition}[\cite{painleveBook} Chapter II Section 1] Let $G$ be a Lie group acting holomorphically on a complex manifold $X$, $\Sigma \subset \C^*$ the support of an oriented path, $J\!:\Sigma \to G$ a map. A Riemann-Hilbert problem (RH problem) with values in $X$ defined by $J$ consists of finding a map $\Phi(t)\!: \C^*\setminus\Sigma \to X$ with the following properties:
\begin{enumerate}
\item $\Phi$ is analytic in $\C^*\setminus\Sigma$;
\item the limits $\Phi_-(t)$ of $\Phi$ from the minus side of $\Sigma$ and the limit $\Phi_+(t)$ from the plus side of $\Sigma$ exist for all $t \in \Sigma$ and are related by
\begin{equation*}
\Phi_+(t) = J(t) \cdot \Phi_-(t)   
\end{equation*}
\item $\Phi(t)$ has prescribed asymptotic behaviour as $t \to 0$.
\end{enumerate}
\end{definition}
A BPS structure $(\Gamma, Z, \Omega)$ induces in a very natural way various RH problems with values in $\mathbb{T}$ and $\Aut(\mathbb{T})$.
\begin{definition}[\cite{bridRH} Section 3.1, \cite{gmn} Section 5.1, \cite{fgs} Section 3.2]\label{definitionRH} The RH problem of a BPS structure $(\Gamma, Z, \Omega)$ with values in $\Aut(\mathbb{T})$ is obtained with the choices $\Sigma = \bigcup_{\g | \Omega(\g) \neq 0} \ell_{\g}$ and $J|_{\ell} = S_{\ell}$ for all rays $\ell \subset \Sigma$. The $t \to 0$ asymptotics are $e^{Z/t} \Phi(t) \to I$, where $Z$ is regarded naturally as a vector field on $\mathbb{T}$ and $I \in \Aut(\mathbb{T})$ is the identity. We define the corresponding RH problem at $\xi$ with values in $\mathbb{T}$ by using the natural action of $\Aut(\mathbb{T})$ on $\mathbb{T}$ and evaluating $\Phi$ at a point $\xi \in \mathbb{T}$. The $t \to 0$ asymptotics are then $e^{Z/t}\Phi(t) \to \xi$.
\end{definition}
\begin{rmk} The main difficulty with this general definition is that $\Sigma \subset \C^*$ might be dense. This does not happen in the finite integral case of course, and in that case $J$ takes values in $\operatorname{Bir}(\mathbb{T})$. In that case one needs to make sure that $\xi$ does not lie in the indeterminacy locus.
\end{rmk}
Composing with twisted characters we define the components
\begin{equation*}
\Phi_{\alpha}(t) = x_{\alpha} \circ \Phi.
\end{equation*}
\begin{definition}[\cite{bridRH} Problem 3.1] The birational RH problem of a finite, integral BPS structure $(\Gamma, Z, \Omega)$ (as in Theorem \ref{bridThmA}) at $\xi \in \mathbb{T}$ is the RH problem in the sense of Definition \ref{definitionRH}, with values in $\mathbb{T}$ and where $J$ takes values in $\operatorname{Bir}(\mathbb{T})$, with the additional constraint that for some $k > 0$ we have for all $\alpha \in \Gamma$ 
\begin{equation*}
|t|^{-k} < |\Phi_{\alpha}(t)| < |t^k|,\,|t|\gg 0.
\end{equation*}
\end{definition}
\begin{definition}[\cite{bridRH} equation 12] Suppose $\Phi$ is a solution to the birational RH problem $(\Gamma, Z, \Omega)$. We define a map $\Psi\!:\C^* \setminus \Sigma \to \mathbb{T}_+$ (as in Theorem \ref{bridThmA}) using the simply transitive action of $\mathbb{T}_+$ on $\mathbb{T}$, by 
\begin{equation*}
e^{Z/t} \Phi = \Psi \cdot \xi,
\end{equation*}
We write $\Psi_{\alpha} = y_{\alpha} \circ \Psi$ for its components. Clearly $\Phi$, $\Psi$ are equivalent data, and we still call $\Psi$ a solution to the birational RH problem.
\end{definition} 
\begin{definition} The functions $\Psi_{\ell, \beta}$ appearing in Theorem \ref{bridThmA} denote the unique analytic continuation to $\H_{\ell}$ of $\Psi_{\beta}$ restricted to a sector between active rays containing the generic ray $\ell$. 
\end{definition}
Next we turn to Frobenius bundles, modelled on Dubrovin's Frobenius manifolds \cite{dubrovin}. Suppose $M$ is a complex manifold.
\begin{definition}[\cite{hert} Definition 5.6]\label{DefFrobType} A \emph{Frobenius bundle} is a holomorphic vector bundle $K \to M$ endowed with data $(\nabla^r, C, \U, \V, g)$, in the holomorphic category, with values in the bundle $K$, where 
\begin{enumerate}
\item[$\bullet$] $\nabla^r$ is a flat connection,
\item[$\bullet$] $C$ is a Higgs field, that is a $1$-form with values in endomorphisms, with $C \wedge C = 0$,
\item[$\bullet$] $\U, \V$ are endomorphisms,
\item[$\bullet$] $g$ is a nondegenerate symmetric bilinear form (called the holomorphic metric, although it is not positive definite),
\end{enumerate}
satisfying the conditions  
\begin{align}\label{FrobTypeCond1}
\nonumber \nabla^r(C) &=0,\\
\nonumber [C, \U] &= 0,\\
\nonumber \nabla^r(\V) &= 0,\\
\nabla^r(\U) - [C, \V] + C &= 0
\end{align}
and the conditions on the metric $g$
\begin{align}\label{FrobTypeCond2}
\nonumber \nabla^r(g) &= 0,\\
\nonumber g(C_X a, b) &= g(a, C_X b),\\
\nonumber g(\U a, b) &= g(a, \U b),\\
g(\V a, b) &= -g(a, \V b).
\end{align}
\end{definition}
It turns out that, under some assumptions on the family, variations of BPS structure are equivalent to certain Frobenius bundles. This construction uses the holomorphic generating function for $\dt$ invariants introduced by Joyce \cite{joy}. Following loc. cit. equation 2 we define combinatorial coefficients $c(\alpha_1, \cdots, \alpha_k) \in \Q$ given by a sum over connected trees $T$ with vertices labelled by $\{1, \ldots, k\}$, endowed with a compatible orientation, 
\begin{equation*} 
c(\alpha_1, \cdots, \alpha_k) = \sum_T \frac{1}{2^{k - 1}}\prod_{\{i \to j\} \subset T} (-1)^{\bra \alpha_i, \alpha_j\ket} \bra \alpha_i, \alpha_j\ket.
\end{equation*}
Fix a basis $\{\beta_j\}$ for $\Gamma$. We introduce a vector of formal parameters ${\bf s}$ with components $s_j$, corresponding to basis elements $\beta_j$. Writing $\alpha \in \Gamma$ as $\alpha = \sum_j a_j \beta_j$ we set ${\bf s}^{|\alpha|} = \prod_{j} s^{|\alpha_j|}_j$.
\begin{thm}[Joyce \cite{joy} Theorem 3.7, \cite{AnnaJ} Proposition 3.17]\label{joyceThm} Suppose $(\Gamma, Z, \Omega)$ is a framed variation of BPS structure over a complex manifold $M$, such that $Z$ takes values in the complement of a line through $0\in \C$. Then there exist essentially unique multi-valued holomorphic functions $J_k\!: (\C^*)^k \to \C^*$ such that 
\begin{align*} 
\nonumber f^{\alpha}(Z) = &\sum_{\alpha_1 + \cdots + \alpha_k = \alpha,\, Z(\alpha_i) \neq 0} c(\alpha_1, \ldots, \alpha_k) J_k(Z(\alpha_1), \ldots , Z(\alpha_k))\\ &\prod_i {\bf s}^{|\alpha_i|} \dt(\alpha_i, Z).
\end{align*} 
is a well-defined formal power series in ${\bf s}$, whose coefficients are holomorphic functions of $Z$. We define the corresponding Joyce holomorphic generating function as the well-defined formal power series in ${\bf s}$ with coefficients in $\C[\Gamma]$ given by
\begin{equation*}
f(Z) = \sum_{\alpha \neq 0} f^{\alpha}(Z) x_{\alpha}.
\end{equation*}
\end{thm}
\begin{rmk} When $(\Gamma, Z, \Omega)$ is uncoupled these results hold without the extra assumption on $Z$, and we give an explicit formula for $f^{\alpha}(Z)$ in Lemma \ref{uncoupledJoyce}.
\end{rmk}
\begin{prop}[\cite{AnnaJ} Proposition 3.17]\label{frobTypeProp} Let $(\Gamma, Z, \Omega)$ be a framed variation of BPS structure over a complex manifold $M$, such that $Z$ takes values in the complement of a line through $0 \in \C$, endowed with the choice of a basis for $\Gamma$. Let $K \to M$ be the trivial infinite-dimensional bundle with fibre $\C[\Gamma]$. Then the choices  
\begin{align*}
\nabla^r &= d + \sum_{\alpha \neq 0} [ f^{\alpha}(Z) x_{\alpha}, -] \frac{d Z(\alpha)}{Z(\alpha)},\\
C &= - dZ,\\
\U &= Z,\\
\V &= [ f(Z) , - ],\\
g(x_{\alpha}, x_{\beta}) &= \delta_{\alpha \beta}.
\end{align*}
satisfy the Frobenius bundle conditions \eqref{FrobTypeCond1}, \eqref{FrobTypeCond2} as formal power series in the variables ${\bf s}$.  
\end{prop}  
Note that here we use the Lie algebra structure on $\C[\Gamma]$ just to describe endomorphims of $K$, i.e. we work with a vector bundle (rather than with a principal bundle as in \cite{joy}).
\begin{rmk}\label{deformation} We may deform the Poisson bracket on $\C[\Gamma]$ by
\begin{equation*}
[x_{\alpha}, x_{\beta}]_{\hbar} = (i\hbar) [x_{\alpha}, x_{\beta}] = (-1)^{\bra \alpha,\beta \ket}(i\hbar)\bra \alpha,\beta \ket x_{\alpha + \beta},
\end{equation*}
and the combinatorial coefficients by
\begin{equation*} 
c_{\hbar}(\alpha_1, \cdots, \alpha_k) = \sum_T \frac{1}{2^{k - 1}}\prod_{\{i \to j\} \subset T} (-1)^{\bra \alpha_i, \alpha_j\ket} (i\hbar\bra \alpha_i, \alpha_j\ket).
\end{equation*}
This is natural from the viewpoint of refined DT theory, see Remark \ref{hbarBody}. Under the assumptions of Theorem \ref{joyceThm} it is possible to find a lift $\dt_{\hbar}\!:\Gamma \to \Q[\hbar]$ (with $\dt = \dt_{\hbar}|_{\hbar = 1}$) such that 
\begin{align*} 
\nonumber f^{\alpha}_{\hbar}(Z) = &\sum_{\alpha_1 + \cdots + \alpha_k = \alpha,\, Z(\alpha_i) \neq 0} c_{\hbar}(\alpha_1, \ldots, \alpha_k) J_k(Z(\alpha_1), \ldots , Z(\alpha_k))\\ &\prod_i {\bf s}^{|\alpha_i|} \dt_{\hbar}(\alpha_i, Z)
\end{align*} 
is a well-defined formal power series in ${\bf s}$, $\hbar$ whose coefficients are holomorphic functions of $Z$. This may be proved e.g. as in \cite{AnnaJ} Proposition 3.17. The $1$-parameter family $K_{\hbar}$ of \eqref{hbarIntro} is then given by the deformations  
\begin{align*}
\nabla^r_{\hbar} &= d + \sum_{\alpha \neq 0} [ f^{\alpha}_{\hbar}(Z) x_{\alpha}, -]_{\hbar} \frac{d Z(\alpha)}{Z(\alpha)},\\
\V &= [ f_{\hbar}(Z) , - ]_{\hbar}. 
\end{align*}
where $f_{\hbar}(Z) = \sum_{\alpha \neq 0} f^{\alpha}_{\hbar}(Z) x_{\alpha}$. 

Clearly in the uncoupled case we have $f^{\alpha}_{\hbar}(Z) = f^{\alpha}(Z)$ so only the Poisson bracket is deformed as above.
\end{rmk}

An advantage of working with Frobenius bundles is that the holomorphic data $(\nabla^r, C, \U, \V)$ can be canonically projected to a subbundle $K' \subset K$ using the metric $g$. This seems especially useful if $K'$ is finite dimensional. However in general the resulting bundle is no longer Frobenius, i.e. the connection $\nabla^r$ is not flat. This construction is studied in detail in \cite{AnnaTomJ}. We will see that the flatness condition for all subbundles characterises uncoupled BPS structures. 
\section{$A_1$ Frobenius bundles}\label{A1Sec1}

In this Section we study a general uncoupled variation of BPS structure $(\Gamma, \Z, \Omega)$ of rank $2$, i.e. with $\operatorname{rk}(\Gamma) = 2$. 

In order to make contact with the material of \cite{bridRH} Section 5.1 we write the charge lattice $\Gamma$ as $\Z \g \bigoplus \Z \gd$ and refer to the rank $2$ uncoupled case as the (double) $A_1$ case. However for us the pairing $\bra \g, \gd \ket$ is arbitrary (while it is fixed to $-1$ in loc. cit.). The BPS spectrum is constant in $Z \in \Hom(\Gamma, \C)$ and vanishes except for $\Omega(\pm \g) = \Omega$. It follows that the DT spectrum vanishes except for 
\begin{equation*}
\dt(\pm k\gamma) = \frac{\Omega}{k^2}.
\end{equation*}

Let $f(Z)$ denote the Joyce holomorphic generating function. In general, as we explained in the previous Section, this is a Laurent series $f(Z) = \sum_{ \alpha \neq 0} f^{\alpha}(Z) x_{\alpha}$ with coefficients in $\C\pow{\bf s}$ (and in fact an element of $\C[\Gamma]\pow{\bf s}$). In the present double $A_1$ case we have formal parameters
\begin{equation*}
s = s_1 = s_{\g},\, s_2 = s_{\gd}.
\end{equation*}
\begin{lem}\label{A1joy} For the double $A_1$ we have for $k \in \Z \setminus \{0\}$
\begin{equation*}
f^{k \g}(Z) =  \frac{1}{2\pi i} \frac{\Omega}{k^2}s^{k}.
\end{equation*} 
(a constant, independent of $Z$) while all the other $f^{\alpha}(Z)$ vanish identically. In particular we have the symmetry $f^{\alpha} = f^{-\alpha}$.
\end{lem}
\begin{proof} The formal power series $f^{\alpha}(Z)$ can be written as a sum over trees $T$ with vertices labelled by charges $\alpha_i$. The contribution of $T$ is weighted by factors of $\prod_i \dt(\alpha_i)$ and $\prod_{i \to j}\bra \alpha_i, \alpha_j\ket$. In the present double $A_1$ case the first factor vanishes unless all the vertices of $T$ are labelled by integral multiples of $\g$. But for such $T$ the second factor vanishes unless there is only a single vertex, labelled by $k \g$. The contribution for this $T$ is the constant $\dt(k \gamma) = \frac{\Omega}{k^2}$.
\end{proof} 
\begin{cor}\label{A1Frob} For the Frobenius type structure of the double $A_1$ we have
\begin{align*}
\nabla^r &= d + \sum_{k \neq 0} \frac{\Omega}{2\pi i k^2} s^k [x_{k\g}, - ] \frac{dZ(k\g)}{Z(k\g)},\\
\V &= \sum_{k \neq 0} \frac{\Omega}{2\pi i k^2} s^k [x_{k\g}, - ].
\end{align*} 
\end{cor}
\begin{proof} This follows at once from Lemma \ref{A1joy} and the general formulae for $\nabla^r$, $\V$ of Proposition \ref{frobTypeProp}.
\end{proof} 
Fix a finite subset $\Delta = \{\alpha_i\} \subset \Gamma$, $i = 1, \ldots, N$. 
\begin{definition} We denote by $K_{\Delta} \subset K$ the rank $N$ subbundle spanned by $\{x_{\alpha_i}\}$, $i = 1, \ldots, N$. We write $\pi\!: K \to K_{\Delta}$ for the orthogonal projection with respect to $g$.
\end{definition}
Note that $K_{\Delta} \subset K$ is preserved by the endomorphism $\U$ and Higgs field $C$. 
\begin{lem}\label{A1projection} The collection of holomorphic objects $(\pi \nabla^r, C, \U, \pi \V, g)$ is a Frobenius type structure on $K_{\Delta}$.
\end{lem}
\begin{proof} Let us first check that the connection $\pi \nabla^r$ is flat. Fixing $i = 1, \ldots, N$ we compute
\begin{align*}
\pi \nabla^r(x_{\alpha_i}) &= \sum_{\alpha \neq 0} \pi\left( f^{\alpha} (-1)^{\bra \alpha, \alpha_i\ket} \bra \alpha, \alpha_i\ket x_{\alpha + \alpha_i} \right) d\log Z(\alpha)\\
&= \sum^N_{j = 1} (-1)^{\bra \alpha_j, \alpha_i\ket} \bra \alpha_j, \alpha_i\ket f^{\alpha_j - \alpha_i} x_{\alpha_j} d\log Z(\alpha_j - \alpha_i).
\end{align*}
So writing $\pi \nabla^r = d + A$ in the frame $x_{\alpha_i}$ we have 
\begin{equation}\label{Amatrix}
A_{ji} = (-1)^{\bra \alpha_j, \alpha_i\ket} \bra \alpha_j, \alpha_i\ket f^{\alpha_j - \alpha_i} d\log Z(\alpha_j - \alpha_i).
\end{equation}
By Lemma \ref{A1joy} $f^{\alpha}$ is constant in $Z$, so the curvature $2$-form $F(A) = dA + A\wedge A$ of $\pi \nabla^r$ is given by
\begin{align}\label{curvature}
\nonumber F(A)_{ji} = &\sum^N_{k = 1} (-1)^{\bra \alpha_j, \alpha_k\ket + \bra \alpha_k, \alpha_i\ket} \bra \alpha_j, \alpha_k\ket \bra \alpha_k, \alpha_i\ket \\& f^{\alpha_j - \alpha_k} f^{\alpha_k - \alpha_i} d\log Z(\alpha_j - \alpha_k)\wedge d\log Z(\alpha_k - \alpha_i).
\end{align}
By Lemma \ref{A1joy} the product $f^{\alpha_j - \alpha_k} f^{\alpha_k - \alpha_i}$ vanishes unless the classes $\alpha_j - \alpha_k$, $\alpha_k - \alpha_i$ are both multiples of $\gamma$. But it that case the $2$-form $d\log Z(\alpha_j - \alpha_k)\wedge d\log Z(\alpha_k - \alpha_i)$ vanishes.

Similarly we check that $\pi \V$ is flat with respect to $\pi \nabla^r$. Fixing $i = 1, \dots, N$ we compute
\begin{align*}
\pi\V(x_{\alpha_i}) &=  \sum_{\alpha \neq 0} \pi\left( f^{\alpha} (-1)^{\bra \alpha, \alpha_i\ket} \bra \alpha, \alpha_i\ket x_{\alpha + \alpha_i} \right)\\
&= \sum^N_{j = 1} (-1)^{\bra \alpha_j, \alpha_i\ket} \bra \alpha_j, \alpha_i\ket f^{\alpha_j - \alpha_i} x_{\alpha_j}. 
\end{align*}
So the matrix $V$ representing $\pi \V$ in the frame $x_{\alpha_i}$ is
\begin{equation}\label{Vmatrix}
V_{ji} = (-1)^{\bra \alpha_j, \alpha_i\ket} \bra \alpha_j, \alpha_i\ket f^{\alpha_j - \alpha_i}. 
\end{equation}  
In particular by Lemma \ref{A1joy} $V$ is constant in $Z$, so in the frame $x_{\alpha_i}$ we have
\begin{equation*}
\pi \nabla^r (\pi \V) = [A, V].
\end{equation*}
Using \eqref{Amatrix}, \eqref{Vmatrix} we compute
\begin{align*}
[A, V]_{kl} = &\sum^N_{p = 1} (-1)^{\bra \alpha_k, \alpha_p \ket + \bra \alpha_p, \alpha_l \ket }\bra \alpha_k, \alpha_p \ket \bra \alpha_p, \alpha_l \ket f^{\alpha_k - \alpha_p} f^{\alpha_p - \alpha_l}\\
&\left(d\log Z(\alpha_k - \alpha_p) - d\log Z(\alpha_p - \alpha_l)\right).
\end{align*} 
By Lemma 1 the product $f^{\alpha_k - \alpha_p} f^{\alpha_p - \alpha_l}$ vanishes unless $\alpha_k - \alpha_p$, $\alpha_p - \alpha_l$ are both multiples of $\gamma$. But in that case we have 
\begin{equation*}
d\log Z(\alpha_k - \alpha_p) = d\log Z(\alpha_p - \alpha_l) = d \log Z(\gamma).
\end{equation*}
So $[A, V]$ vanishes identically. Checking the other conditions for a Frobenius type structure is straightforward. 
\end{proof}
\begin{definition} In the following we call the structure $(\pi \nabla^r, C, \U, \pi \V, g)$ evaluated at the natural point $s = 1$ the Frobenius type structure on $K_{\Delta}$.
\end{definition}
To the Frobenius type structure on $K_{\Delta}$ we can associate a family of meromorphic connections on the trivial rank $N$ holomorphic bundle over $\PP^1$, parametrised by $Z \in M$. This is given by
\begin{equation}\label{isoConn}
\nabla(Z) = d + \left(\frac{U(Z)}{t^2} - \frac{V}{t}\right)dt
\end{equation} 
where $U$, $V$ are the $N \times N$ matrices representing $\U$, $\pi \V$ with respect to the frame $x_{\alpha_i}$. In particular $V$ is a constant skew-symmetric matrix, independent of $Z$. 
\begin{definition}[\cite{hert} Definition 5.6 and Theorem 5.7] The meromorphic connections $\nabla(Z)$ of the Frobenius type structure $K_{\Delta}$ are the meromorphic connections \eqref{isoConn}, depending on $Z$. 
\end{definition}
\begin{lem}\label{isoConnCoeff} We have
\begin{align*}
U(Z)_{ij} = Z(\alpha_i) \delta_{ij},\,
V_{ij} = (-1)^{\bra \alpha_i, \alpha_j\ket} \bra \alpha_i, \alpha_j\ket f^{\alpha_i - \alpha_j}.
\end{align*}
In particular $U(Z)$ is diagonal and $V$ is skew-symmetric.
\begin{proof} The expression for $U(Z)$ follows at once from $\U(Z) = Z$. The expression for $V$ is \eqref{Vmatrix}. 
\end{proof}
\end{lem}
The simplest nontrivial Frobenius bundle $K_{\Delta}$ contained in $K$ has rank $N = 2$ and is given by the following example. 
\begin{exm} Let $\Delta^1 = \{\alpha_1, \alpha_2\} = \{\g + \gd, \gd \}$. Then we have
\begin{align*}\label{basicExm}
\nabla(Z) = d + &\left( 
\frac{1}{t^2} \left(\begin{matrix} Z(\g + \gd) &0\\ 0 & Z(\g)\end{matrix}\right)\right.\\
&\left.- \frac{1}{t}\frac{\bra \g,\gd\ket}{2\pi i} \Omega \left(\begin{matrix} 0 & (-1)^{\bra\g,\gd\ket} \\ -(-1)^{\bra\g,\gd\ket}1 & 0\end{matrix}\right)\right) dt.
\end{align*}
\end{exm}
The previous Example can be immediately generalised.
\begin{lem}\label{oscillatorLem} For all $k \geq 1$ choose $\Delta^k = \{\alpha_{1}, \ldots \alpha_{2k}\}$ with $\{\alpha_{2i - 1}, \alpha_{2i}\} = \{i (\g + \gd), i \gd\}$. Then the meromorphic connection of the Frobenius type structure on $K_{\Delta^k}$ has $U$, $V$ block diagonal, with blocks
\begin{align*}
U^{(i)} &= \left(\begin{matrix}
i Z(\g + \gd) & 0\\ 0 & i Z(\gd) 
\end{matrix}\right),\\ V^{(i)} &= 
\frac{\bra \g, \gd\ket}{2\pi \sqrt{-1}}\Omega\left(\begin{matrix}
0 & (-1)^{i\bra\g,\gd\ket}\\ -(-1)^{i\bra\g,\gd\ket} & 0 
\end{matrix}\right)
\end{align*}
for $i = 1, \cdots, k$. The rank of $K_{\Delta^k}$ is $N = 2k$. 
\end{lem} 
\begin{proof} We only need to check that $V$ is block diagonal with blocks $V^{(i)}$ as above, for $i = 1, \cdots, k$. According to Lemma \ref{isoConnCoeff} for all $k, l$ we have $V_{k l} =  (-1)^{\bra \alpha_k, \alpha_l\ket} \bra \alpha_k, \alpha_l\ket f^{\alpha_k - \alpha_l}$. We compute
\begin{align*}
f^{\alpha_{2i} - \alpha_{2j}} &= f^{(i-j)\gd} = 0,\,
f^{\alpha_{2i - 1} - \alpha_{2j}} = f^{i\g + (i - j)\gd} = \frac{1}{2\pi \sqrt{-1} i^2} \delta_{ij} \Omega,\\
f^{\alpha_{2i - 1} - \alpha_{2j-i}} &= f^{(i-j)(\g +\gd)} = 0.
\end{align*}
It follows that $V_{kl}$ vanishes except for
\begin{align*}
V_{(2i-1) (2i)} &= (-1)^{\bra i(\g + \gd), i\gd\ket}\bra i(\g + \gd), i\gd \ket f^{i\g + (i - j)\gd}\\& = (-1)^{i\bra\g,\gd\ket}\frac{\bra\g,\gd\ket}{2\pi \sqrt{-1}}\Omega,\\
V_{(2i)(2i-1)} &= - V_{(2i-1) (2i)} = -(-1)^{i\bra\g,\gd\ket}\frac{\bra\g,\gd\ket}{2\pi \sqrt{-1}}\Omega.
\end{align*}
\end{proof}
The bundle $K_{\Delta^k}$ of the previous Lemma is the simplest nontrivial rank $N = 2k$ Frobenius type contained in $K$.  
\begin{definition} For all even $N > 0$ the $A_1$ simple oscillator of rank $N$ is the Frobenius type structure $K_{\Delta^{N/2}} \subset K$ given in Lemma \ref{oscillatorLem}.   
\end{definition}

We will study $K_{\Delta^{N/2}}$ in more detail in the next Section. Let us go back to a general Frobenius type structure $K_{\Delta} \subset K$. 
\begin{lem} The generalised monodromy of the meromorphic connections $\nabla(Z)$ of $K_{\Delta}$ is constant in $Z$.
\end{lem}
\begin{proof} This is a standard result for the family of meromorphic connections underlying a Frobenius type structure, see e.g. \cite{bt_stokes} Section 3.3 and \cite{hert} Section 5.2 (based on Dubrovin \cite{dubrovin}).
\end{proof}
We close this Section by giving a standard formula for the generalised monodromy of $\nabla(Z)$ (i.e. its Stokes factors, see e.g. \cite{bt_stokes} Section 2). In particular this shows explicitly that the Stokes factors are constant in $Z$. 

There is a classical formula for the Stokes factors of a linear connection of the form 
\begin{equation*}
d - \left( \frac{\Lambda}{t^2} + \frac{f}{t} \right) dt,
\end{equation*}
where $\Lambda$ is diagonal and $f$ is off-diagonal, in terms of periods, see e.g. \cite{bt_stokes} Theorem 4.5. Periods appear here in the guise of multilogarithms, i.e. the iterated integrals 
\begin{align*}
&M_n(w_1, \ldots, w_n)\\ 
&= (-2\pi i)^n \int_{[0, w_1 + \cdots + w_n]} \frac{dt}{t - w_1} \circ \cdots \circ \frac{dt}{t - (w_1+ \cdots + w_{n-1})}
\end{align*}
(see e.g. \cite{bt_stokes} Section 7). 
\begin{rmk}\label{hyperlogs} The functions $M_n(w_1, \ldots, w_n)$ are also known as hyperlogarithms, see e.g. \cite{gonch} Section 2 where these are defined as the multi-valued functions 
\begin{equation*}
I(a_1 : \ldots : a_{m+1}) = \int^{a_{m+1}}_0 \frac{dt}{t - a_1} \circ \cdots \circ \frac{dt}{t - a_m}.
\end{equation*}
In particular we have
\begin{equation*}
M_n(w_1, \ldots, w_n) = (-2\pi i)^n I(w_1 : w_1 + w_2 : \ldots : w_1 + \cdots + w_n). 
\end{equation*}
According to loc. cit. $I(a_1 : \ldots : a_{m+1})$ is invariant under the affine transformations $a_i \mapsto \lambda a_i + \beta$, so in particular we have
\begin{equation*}
M_n(\lambda w_1, \ldots, \lambda w_n) = M_n(w_1, \ldots, w_n).
\end{equation*}
\end{rmk}

We apply the classical formula to the connection $\nabla(Z)$, of the form
\begin{equation*}
d - \left(\frac{-U(Z)}{t^2} + \frac{V}{t}\right)dt.
\end{equation*}
Note that according to Lemma \ref{isoConnCoeff} $- U(Z)$ is diagonal, with ordered eigenvalues $-Z(\alpha_i)$, and $V$ is off-diagonal. 
\begin{definition}\label{mFunction} We introduce a function $m\!: \Delta \times \Delta \to \Z$ such that $m(\alpha_i, \alpha_j)$ equals $m$ if $\alpha_i - \alpha_j = m\gamma$ for $m \in \Z$, while $m(\alpha_i, \alpha_j) = 0$ if $\alpha_i - \alpha_j$ is not a multiple of $\gamma$. 
\end{definition}
In the following we write $E_{ij}$ to denote the elementary matrix with $(E_{ij})_{kl} = \delta_{ik}\delta_{jl}$ and $I$ for the identity matrix.
\begin{lem} Let $\ell = \pm \R_{> 0} Z(\gamma)$. Consider all sequences $1 \leq i_1 \neq i_2 \neq \cdots \neq i_{n+1} \leq N$ with $Z(\alpha_{i_{n+1}} - \alpha_1) \in \ell$ and $n \geq 0$. Then the Stokes factor $S_{\ell}$ for the connection $\nabla(Z)$ of \eqref{isoConn} is the sum of all products of the form
\begin{align*}
&M_n(m(\alpha_{i_2}, \alpha_{i_1}), m(\alpha_{i_3}, \alpha_{i_2}) \ldots, m(\alpha_{i_{n+1}}, \alpha_{i_n}))\\
& (-1)^{m(\alpha_{i_1}, \alpha_{i_{2}})\bra \gamma, \alpha_{i_{2}}\ket} m(\alpha_{i_1}, \alpha_{i_{2}})\bra \gamma, \alpha_{i_{2}}\ket f^{m(\alpha_{i_1}, \alpha_{i_{2}})\gamma}\\
& (-1)^{m(\alpha_{i_2}, \alpha_{i_{3}})\bra \gamma, \alpha_{i_{3}}\ket} m(\alpha_{i_2}, \alpha_{i_{3}})\bra \gamma, \alpha_{i_{3}}\ket f^{m(\alpha_{i_2}, \alpha_{i_{3}})\gamma}\\
&\cdots\\
&(-1)^{m(\alpha_{i_n}, \alpha_{i_{n+1}})\bra \gamma, \alpha_{i_{n+1}}\ket} m(\alpha_{i_n}, \alpha_{i_{n+1}})\bra \gamma, \alpha_{i_{n+1}}\ket f^{m(\alpha_{i_n}, \alpha_{i_{n+1}})\gamma} E_{i_1 i_{n+1}}
\end{align*}
where the empty product corresponding to $n = 0$ conventionally equals $I$. All the other Stokes factors are trivial. In particular the Stokes factors of $\nabla(Z)$ are constant in $Z$.
\end{lem}
\begin{proof} Let $\ell = \R_{> 0} Z(\alpha_j - \alpha_i)$ be any potential Stokes ray, i.e. the ray spanned by a difference of eigenvalues of $-U(Z)$. The general formula then shows that the Stokes factor attached to $\ell$ is a sum of contributions 
\begin{equation}\label{multilogSell}
M_n(Z(\alpha_{i_2} - \alpha_{i_1}), \ldots, Z(\alpha_{i_{n+1}} - \alpha_{i_n})) V_{i_1 i_2} \cdots V_{i_n i_{n+1}} E_{i_1 i_{n+1}},
\end{equation}
for each sequence $1 \leq i_1 \neq i_2 \neq \cdots \neq i_{n+1} \leq N$ with $Z(\alpha_{i_{n+1}} - \alpha_1) \in \ell$. Here $n \geq 0$ is arbitrary, and the term corresponding to $n = 0$ conventionally equals $I$. Lemma \ref{isoConnCoeff} shows \begin{equation*}
V_{i_k i_{k+1}} = (-1)^{\bra \alpha_{i_k}, \alpha_{i_{k+1}}\ket} \bra \alpha_{i_k}, \alpha_{i_{k+1}}\ket f^{\alpha_{i_k} - \alpha_{i_{k+1}}},
\end{equation*}
and according to Lemma \ref{A1joy} this vanishes unless $\alpha_{i_k} - \alpha_{i_{k+1}}$ is a multiple of $\gamma$. It follows that the contribution \eqref{multilogSell} to $S_{\ell}$ can be written as 
\begin{align*}
&M_n(m(\alpha_{i_2}, \alpha_{i_1}) Z(\gamma), \ldots, m(\alpha_{i_{n+1}}, \alpha_{i_n})Z(\gamma))\\
& (-1)^{m(\alpha_{i_1}, \alpha_{i_{2}})\bra \gamma, \alpha_{i_{2}}\ket} m(\alpha_{i_1}, \alpha_{i_{2}})\bra \gamma, \alpha_{i_{2}}\ket f^{\alpha_{i_1} - \alpha_{i_{2}}} \cdots\\
&(-1)^{m(\alpha_{i_n}, \alpha_{i_{n+1}})\bra \gamma, \alpha_{i_{n+1}}\ket} m(\alpha_{i_n}, \alpha_{i_{n+1}})\bra \gamma, \alpha_{i_{n+1}}\ket f^{\alpha_{i_n} - \alpha_{i_{n+1}}} E_{i_1 i_{n+1}}.  
\end{align*}
By Remark \ref{hyperlogs} we have
\begin{align*}
&M_n(m(\alpha_{i_2}, \alpha_{i_1}) Z(\gamma), \ldots, m(\alpha_{i_{n+1}}, \alpha_{i_n})Z(\gamma)) \\
& = M_n(m(\alpha_{i_2}, \alpha_{i_1}), \ldots, m(\alpha_{i_{n+1}}, \alpha_{i_n})).
\end{align*}
By Definition \ref{mFunction} and Lemma \ref{A1joy} we have 
\begin{equation*}
f^{\alpha_{i_n} - \alpha_{i_{n+1}}} = f^{m(\alpha_{i_n}, \alpha_{i_{n+1}})\gamma}. 
\end{equation*}
Finally we see that the general contribution \eqref{multilogSell} to $S_{\ell}$ vanishes unless all $\alpha_{i_{k}} - \alpha_{i_{k+1}}$ are (nonzero) multiples of $\gamma$. But then $\alpha_{i_1} - \alpha_{i_{n+1}}$ is also a multiple of $\gamma$, i.e. $\ell$ must be one of the rays $\pm Z(\gamma)$. The Lemma follows. 
\end{proof}
\section{$A_1$ simple oscillators}\label{A1Sec2}
In the present Section we collect some (rather standard) computations for the rank $N$ Frobenius bundles $K_{\Delta^{N/2}} \subset K$ contained in the double $A_1$ infinite dimensional Frobenius type structure, i.e. our $A_1$ simple oscillators. Recall from Lemma \ref{oscillatorLem} that the meromorphic connection $\nabla$ of $K_{\Delta^{N/2}}$ is a direct sum 
\begin{equation*}
\nabla = \bigoplus_{m \geq 1} \nabla^{(m)} = \bigoplus_{m \geq 1} d + (t^{-2} U^{(m)} - t^{-1}V^{(m)})dt,
\end{equation*} 
where
\begin{align*}
U^{(m)} &= \left(\begin{matrix}
m Z(\g + \gd) & 0\\ 0 & m Z(\gd) 
\end{matrix}\right),\\ V^{(m)} &= 
\frac{\bra\g,\gd\ket}{2\pi i}\Omega\left(\begin{matrix}
0 & (-1)^{m\bra\g,\gd\ket}\\ -(-1)^{m\bra\g,\gd\ket} & 0 
\end{matrix}\right).
\end{align*}

\begin{lem}\label{A1Stokes} The Stokes rays of $\nabla^{(m)}$ are $\pm\ell_{\g}$. The corresponding Stokes factors are given by
\begin{align*}
S_{\ell_{\g}} &= \left(\begin{matrix} 
1 & 2 i \sinh\left(-\frac{(-1)^{m\bra\g,\gd\ket}}{2}\bra \g,\gd \ket \Omega\right)\\
0 &1
\end{matrix}
\right),\\
S_{-\ell_{\g}} &= \left(\begin{matrix} 
1 & 0\\
-2 i \sinh\left(-\frac{(-1)^{m\bra\g,\gd\ket}}{2}\bra \g,\gd \ket \Omega\right) &1
\end{matrix}
\right).
\end{align*}
\end{lem}
\begin{proof} We use standard Fourier-Laplace methods, see e.g. \cite{bt_stokes} Section 8. By definition the Fourier-Laplace transform of $\nabla^{(m)}$ is the Fuchsian connection with simple poles at $z_1, z_2$
\begin{equation*}
\widehat{\nabla} = d - \left(\frac{A_1}{z-z_1} + \frac{A_2}{z-z_2}\right) dz 
\end{equation*}
where we set $z_1 = - m Z(\g + \gd), z_2 = - m Z(\gd)$, and the nilpotent residues are given by
\begin{align*}
A_1 &= -\frac{(-1)^{m\bra\g,\gd\ket}\bra \g, \gd\ket}{2\pi i}\Omega\left(\begin{matrix} 0 & -1\\ 0 & 0\end{matrix}\right),\\ 
\quad A_2 &= -\frac{(-1)^{m\bra\g,\gd\ket}\bra \g, \gd\ket}{2\pi i}\Omega\left(\begin{matrix} 0 & 0\\ 1 & 0\end{matrix}\right).
\end{align*} 
Suppose $\phi(z) = \left(\begin{matrix} u \\ v \end{matrix}\right)$ is a horizontal section of $\widehat{\nabla}$. Then 
\begin{equation*}
\tilde{\phi}(z) = \phi((z_2 - z_1)z + z_1) = \left(\begin{matrix} \tilde{u}(z) \\ \tilde{v}(z) \end{matrix}\right)
\end{equation*}
solves $\frac{d}{dz}\left(\begin{matrix} \tilde{u} \\ \tilde{v} \end{matrix}\right) = -\frac{(-1)^{m\bra\g,\gd\ket}\bra \g, \gd\ket}{2\pi i}\Omega\left(\begin{matrix} -\frac{1}{z}\,\tilde{v} \\ \frac{1}{z-1}\,\tilde{u} \end{matrix}\right)$ and so we have
\begin{equation*}
z(1-z)\frac{d^2}{dz^2} \tilde{u} + (1-z)\frac{d}{dz}\tilde{u} + \left(\frac{\bra\g,\gd\ket}{2\pi}\right)^2\Omega^2 \tilde{u} = 0,
\end{equation*}
a standard hypergeometric equation 
\begin{equation*}
z(1-z)\frac{d^2}{dz^2} \tilde{u} + (c -(a + b +1) z)\frac{d}{dz}\tilde{u} - (ab) \tilde{u} = 0
\end{equation*}
with parameters
\begin{align*}
a &= i V_{21} = -(-1)^{m\bra\g,\gd\ket}\frac{\bra \g,\gd\ket}{2\pi}\Omega,\\  b &= - a = i V_{12},\\  c &= 1.
\end{align*}
So the unique solution $\tilde{\phi}^{(0)}(z)$ at $z = 0$ with $\tilde{\phi}^{(0)}(0) = e_1 = \left(\begin{matrix} 1 \\ 0 \end{matrix}\right)$ is given in terms of Gauss hypergeometric functions as
\begin{equation}\label{hypergeo}
\tilde{\phi}^{(0)}(z) = \left(\begin{matrix} {}_2 F_1(-a, a, 1; z) \\ \frac{2\pi i}{(-1)^{m\bra\g,\gd\ket}\bra \g,\gd\ket} \frac{1}{\Omega}z \frac{d}{dz} {}_2F_1(-a, a, 1; z) \end{matrix}\right)
\end{equation}
and similarly the unique solution $\tilde{\phi}^{(1)}(z)$ at $z = 1$ with $\tilde{\phi}^{(1)}(1) = e_2 = \left(\begin{matrix} 0 \\ 1 \end{matrix}\right)$ is given by
\begin{equation*}
\tilde{\phi}^{(1)}(z) = \left(\begin{matrix} -\frac{(-1)^{m\bra\g,\gd\ket}\bra \g,\gd\ket}{2\pi i}\Omega(1-z) {}_2 F_1(1-a, 1+a, 2;1-z)\\ - z \frac{d}{dz}(1-z) {}_2 F_1(1-a, 1+a, 2;1-z) \end{matrix}\right)
\end{equation*}
(see \cite{erdelyi} Chapter II Section 2.1). It is well-known that the Fourier-Laplace transform allows to express Stokes factors for $\nabla^{(m)}$ in terms of the analytic continuation of solutions to $\widehat{\nabla}$, see e.g. \cite{bt_stokes} Section 9. In particular applying the formulae in loc. cit. Section 9.2 we find
\begin{equation*}
S_{\ell_{\g}} =  \left(\begin{matrix}  1 & 2\pi iV_{21} (\tilde{\phi}^{(0)}(1))_{1}\\ 0 & 1 \end{matrix}\right),\,S_{-\ell_{\g}} =  \left(\begin{matrix}  1 & 0\\ 2\pi iV_{12} (\tilde{\phi}^{(1)}(0))_{2} & 1 \end{matrix}\right). 
\end{equation*} 
On the other hand we have 
\begin{align*}
\tilde{\phi}^{(0)}(1) &= {}_2 F_1(-a, a, 1; z)\text{ by } \eqref{hypergeo}\\ 
&= \frac{1}{\Gamma(1 - a)\Gamma(1 + a)} \text{ (see \cite{whit} Chap XIV p. 282)}\\
&= \frac{\sin(\pi a)}{\pi a} \text{ (by Euler reflection \cite{erdelyi} Chap. I Sec. 1.2 equ. 8)}.  
\end{align*}
Using the relation $a = i V_{21}$ gives the result for $S_{\ell_{\g}}$. The computation for $S_{-\ell_{\g}}$ is completely analogous.
\end{proof}
Let $Y^{(m)}(t) = Y^{(m)}_{ij}(t)$ be the $GL(2, \C)$ fundamental solution to $\nabla^{(m)}$. Define 
\begin{equation}\label{fundSol}
\Psi^{(m)}_{ij}(t) = e^{Z_j/t} Y^{(m)}_{ij}(t)
\end{equation}
where $Z_1 = m Z(\g + \gd)$, $Z_2 = m Z(\gd)$. Recall $Y^{(m)}(t)$ is characterised by the asymptotics $\Psi^{(m)}(t) \to I$ as $t \to 0$ in a sector.
\begin{lem}\label{integralEqu} The functions $\Psi^{(m)}_{ij}(t)$ satisfy the integral equations
\begin{align*}
\Psi^{(m)}_{11}(t) &= 1 - \eta \int_{-\ell_{\g}} \frac{dt'}{t'}\frac{t}{t' - t} \Psi^{(m)}_{12}(t')  e^{mZ(\g)/t'},\\
\Psi^{(m)}_{12}(t) &= \eta  \int_{\ell_{\g}} \frac{dt'}{t'}\frac{t}{t' - t} \Psi^{(m)}_{11}(t')  e^{-mZ(\g)/t'},\\
\Psi^{(m)}_{21}(t) &= -\eta  \int_{-\ell_{\g}} \frac{dt'}{t'}\frac{t}{t' - t} \Psi^{(m)}_{22}(t')  e^{mZ(\g)/t'},\\
\Psi^{(m)}_{22}(t) &= 1 + \eta \int_{\ell_{\g}} \frac{dt'}{t'}\frac{t}{t' - t} \Psi^{(m)}_{21}(t')  e^{-mZ(\g)/t'}
\end{align*}
where
\begin{equation*}
\eta = \frac{1}{2\pi i}(S_{\ell_{\g}})_{12} = -\frac{1}{2\pi i}(S_{-\ell_{\g}})_{21} = \frac{1}{\pi} \sinh\left(-\frac{(-1)^{m\bra \g,\gd\ket}}{2}\bra \g,\gd \ket \Omega\right).
\end{equation*}
\end{lem}
\begin{proof} The function $\Psi^{(m)}(t)$ is uniquely characterised as the solution to a Riemann-Hilbert factorisation problem with rays $\pm \ell_{\gamma}$, jumps $S_{\pm \ell_{\gamma}}$ as in Lemma \ref{A1Stokes}, asymptotics $\Psi^{(m)}(t) \to I$ as $t \to 0$ in a sector and polynomial growth as $t \to \infty$. Standard results allow to recast this Riemann-Hilbert problem in terms of integral equations as claimed, see e.g. \cite{painleveBook} Chapter 3, Section 1. Our present application is in fact a limiting case of \cite{dubrovinFusion} Proposition 2.2. A reference which is very close to our notation is \cite{gmn} Appendix C. Indeed our function $\Psi(t)$ is precisely the function $\Phi(x)$ appearing in loc. cit. equation 6, evaluated at $x= \beta t^{-1}$ and in the limit $\beta \to 0$, with parameters $\mu_{12} = -\mu_{21}=\eta$. Note that the change of variable $y = \beta t'^{-1}$, $x = \beta t^{-1}$ turns the integral kernel $dy(y-x)^{-1}$ appearing in loc. cit. equation 6 into our kernel $t dt' (t'(t' - t))^{-1}$.
\end{proof}
\section{$A_1$ large $N$ limit}\label{A1Sec3}
We continue our study of the rank $N$ simple oscillator $K_{\Delta^{N/2}} \subset K$. We regard the Frobenius bundle structure on $K_{\Delta^{N/2}}$ as depending on the free parameter $\bra \g, \gd \ket$ via the formulae of Lemma \ref{oscillatorLem}.
\begin{definition}\label{hbarK} Let $\hbar \in \R_{> 0}$. The rescaled simple oscillator $K_{\Delta^{N/2}, \hbar}$ is obtained by replacing
\begin{equation}\label{rescaling}
 \bra \g, \gd\ket \Omega \mapsto (\sqrt{-1}\hbar)\bra \g, \gd\ket \Omega
\end{equation}
in the formulae of Lemma \ref{oscillatorLem}. 
\end{definition}
In other words $K_{\Delta^{N/2}, \hbar}$ is the projection of the deformed bundle $K_{\hbar}$ discussed in Remark \ref{deformation}.
\begin{rmk}\label{hbarBody} Perhaps the best way to motivate the rescaling \eqref{rescaling} is through refined Donalson-Thomas invariants. In the refined theory one deforms the commutative algebra structure on $\C[\Gamma]$ to a non-commutative product
\begin{equation*}
x_{\g} *_q x_{\gd} = q^{\frac{1}{2}\bra \g,\gd\ket} x_{\g + \gd}.
\end{equation*} 
The Lie bracket becomes simply the commutator, 
\begin{equation*}
[x_{\g}, x_{\gd}]_q = (q^{\frac{1}{2}\bra \g,\gd\ket} - q^{-\frac{1}{2}\bra \g,\gd\ket})  x_{\g + \gd}.
\end{equation*}
Setting $q = -e^{i \hbar}$ we find
\begin{equation*}
[x_{\g}, x_{\gd}]_q = (-1)^{\bra\g,\gd\ket}(\bra \g, \gd \ket i\hbar) x_{\g + \gd} + O(\hbar^2).
\end{equation*}
So to first order the non-commutative deformation required for the refined theory is given by the rescaling \eqref{rescaling}.
\end{rmk}
By Definition \ref{hbarK} the meromorphic connection $\nabla_{\hbar}$ of $K_{\Delta^{N/2}, \hbar}$ splits just as before
\begin{equation*}
\nabla_{\hbar} = \bigoplus_{m \geq 1} \nabla^{(m)}_{\hbar} = \bigoplus_{m \geq 1} d + (t^{-2} U^{(m)} - t^{-1}V^{(m)}_{\hbar})dt,
\end{equation*} 
where
\begin{equation*}
V^{(m)}_{\hbar} = 
\frac{\bra\g,\gd\ket \hbar}{2\pi}\Omega\left(\begin{matrix}
0 & (-1)^{m\bra\g, \gd\ket}\\ -(-1)^{m\bra\g, \gd\ket} & 0 
\end{matrix}\right).
\end{equation*}
\begin{lem}\label{A1Stokeshbar} The Stokes rays of $\nabla^{(m)}_{\hbar}$ are $\pm\ell_{\g}$. The corresponding Stokes factors are given by
\begin{align*}
S_{\ell_{\g}, \hbar} &= \left(\begin{matrix} 
1 & 2 i \sinh\left(-\frac{(-1)^{m\bra\g, \gd\ket}}{2}\bra \g,\gd \ket i\hbar \Omega\right)\\
0 &1
\end{matrix}
\right)\\
&= \left(\begin{matrix} 
1 & (-1)^{m\bra\g, \gd\ket} \bra \g,\gd \ket \hbar \Omega\\
0 &1
\end{matrix}
\right) + O(\hbar^2),\\
S_{-\ell_{\g}, \hbar} &= \left(\begin{matrix} 
1 & 0\\
-2 i \sinh\left(-\frac{(-1)^{m\bra\g, \gd\ket}}{2}\bra \g,\gd \ket i\hbar \Omega\right) &1
\end{matrix}
\right)\\
&= \left(\begin{matrix} 
1 & 0\\
-(-1)^{m\bra\g, \gd\ket} \bra \g,\gd \ket \hbar \Omega &1
\end{matrix}
\right) + O(\hbar^2).
\end{align*}
\end{lem}
\begin{proof} The result follows at once from Lemma \ref{A1Stokes}.
\end{proof}
Let $Y^{(m)}_{\hbar}(t) = Y^{(m)}_{\hbar, ij}(t)$ be the $GL(2, \C)$ fundamental solution to $\nabla^{(m)}_{\hbar}$. Define 
\begin{equation} 
\Psi^{(m)}_{\hbar, ij}(t) = e^{Z_j/t} Y^{(m)}_{\hbar, ij}(t)
\end{equation}
where $Z_1 = m Z(\g + \gd)$, $Z_2 = m Z(\gd)$. 
\begin{lem}\label{integralEquhbar} The functions $\Psi^{(m)}_{\hbar, ij}(t)$ satisfy  
\begin{align*}
\Psi^{(m)}_{\hbar, 11}(t) &= 1 + O(\hbar^2),\\
\Psi^{(m)}_{\hbar, 12}(t) &= \frac{1}{2\pi i} (-1)^{m\bra\g, \gd\ket} \bra \g,\gd \ket \hbar \Omega \int_{\ell_{\g}} \frac{dt'}{t'}\frac{t}{t' - t} e^{-mZ(\g)/t'} + O(\hbar^2),\\
\Psi^{(m)}_{\hbar, 21}(t) &= -\frac{1}{2\pi i} (-1)^{m\bra\g, \gd\ket} \bra \g,\gd \ket \hbar \Omega  \int_{-\ell_{\g}} \frac{dt'}{t'}\frac{t}{t' - t} e^{mZ(\g)/t'} + O(\hbar^2),\\
\Psi^{(m)}_{\hbar, 22}(t) &= 1 + O(\hbar^2).
\end{align*}
\end{lem}
\begin{proof} By Lemma \ref{A1Stokeshbar} the function $\Psi^{(m)}_{ij, \hbar}(t)$ satisfy equations identical to those of Lemma \ref{integralEqu}, with $\eta$ replaced by 
\begin{equation*}
\frac{1}{2\pi i}(S_{\ell_{\g}, \hbar})_{12} = -\frac{1}{2\pi i}(S_{-\ell_{\g}, \hbar})_{21} = \frac{1}{\pi} \sinh\left(-\frac{(-1)^{m\bra\g, \gd\ket}}{2}\bra \g,\gd \ket i\hbar \Omega\right).
\end{equation*}
Expanding around $\hbar = 0$ gives the result.
\end{proof}
\begin{cor} We have
\begin{equation*}
\log \Psi^{(m)}_{\hbar} = \left(
\begin{matrix}
0 & \delta^{(m)}(t)\\
-\delta^{(m)}(-t) & 0
\end{matrix}
\right) + O(\hbar^2)
\end{equation*}
where
\begin{equation*}
\delta^{(m)}(t) = \frac{1}{2\pi i} (-1)^{m\bra\g, \gd\ket} \bra \g,\gd \ket \hbar \Omega \int_{\ell_{\g}} \frac{dt'}{t'}\frac{t}{t' - t} e^{-mZ(\g)/t'}. 
\end{equation*}
\end{cor}
\begin{proof} The result follows from Lemma \ref{integralEquhbar}, by making the change of variable $t' \mapsto -t'$ in the integral for $\Psi^{(m)}_{21, \hbar}(t)$.
\end{proof}
\begin{cor} We have
\begin{align*}
\log \Psi^{(m)}_{\hbar}(t)_{(12)} &=  (-1)^{m\bra\g, \gd\ket} \bra \g,\gd \ket \hbar \Omega \frac{t}{\pi i} \int_{\ell_{\g}} \frac{dt'}{(t')^2 - t^2} e^{-mZ(\g)/t'}\\&+ O(\hbar^2).
\end{align*}
\end{cor}
\begin{proof} Following the notation of the previous Lemma we have 
\begin{equation*}
(\log \Psi^{(m)}_{\hbar}(t))_{12} + (\log \Psi^{(m)}_{\hbar}(t))_{21} = \delta^{(m)}(t) - \delta^{(m)}(-t).
\end{equation*}
So the claim follows from a straightforward calculation. 
\end{proof}
As in the Introduction we let $\hat{\xi}$  denote the vector $\left(\begin{matrix}1 \\ 1\end{matrix}\right)\in\C^2$ and $\Pi$ be the linear functional on $\C^2$ given by $\Pi(w_1, w_2) = w_1 + w_2$.
\begin{prop}\label{A1largeN} We have 
\begin{align*}
\exp\left(\frac{1}{\hbar}\sum^{\infty}_{m = 1} \frac{(-1)^{m\bra\g, \gd\ket}}{m}  \Pi \log \Psi^{(m)}_{\hbar}((2\pi)^{-1} i t)\hat{\xi} \right) = \Psi_{\gd}(t) + O(\hbar).
\end{align*}
\end{prop}
\begin{proof} In this proof we write $Z = Z(\g)$ for brevity. Note that we have $\Pi \log \Psi^{(m)}_{\hbar}((2\pi)^{-1} i t)\hat{\xi} = \log \Psi^{(m)}_{\hbar}((2\pi)^{-1} i t)_{(12)}$. By the previous Lemma we have
\begin{align*}
(\log \Psi^{(m)}_{\hbar}(i t))_{12} + (\log \Psi^{(m)}_{\hbar}(i t))_{21} &=  (-1)^{m\bra\g, \gd\ket} \bra \g,\gd \ket \hbar \Omega\\ &\frac{t}{\pi } \int_{\ell_{\g}} \frac{dt'}{(t')^2 + t^2} e^{-mZ/t'}+ O(\hbar^2).
\end{align*}
By the definition of $\ell_{\gamma}$ we are integrating over $t' = Z s$, $s > 0$, so we have
\begin{align*}
\frac{t}{\pi } \int_{\ell_{\g}} \frac{dt'}{(t')^2 + t^2} e^{-mZ(\g)/t'} &= \frac{1}{\pi} \int^{\infty}_0 \left(\frac{Z}{t}\right)\frac{ds}{(\frac{Z}{t})^2 s^2 + 1} e^{-m /s}\\
&= \frac{1}{\pi} \int^{\infty}_0 \left(\frac{Z}{t}\right)^{-1}\frac{ds}{(\frac{Z}{t})^{-2} s^2 + 1} e^{-m s}
\end{align*}
(using the change of variable $s \mapsto s^{-1}$). The right hand side can be rewritten as
\begin{equation*}
\frac{1}{\pi} \int^{\infty}_0  e^{-m s} \frac{d}{ds}\arctan\left(\left(\frac{Z}{t}\right)^{-1} s\right) ds 
\end{equation*}
and so integrating by parts as
\begin{equation*}
-m\frac{1}{\pi} \int^{\infty}_0 \arctan\left(\left(\frac{Z}{t}\right)^{-1} s\right)  e^{-m s} ds.
\end{equation*}
By these identities we can rewrite the series  
\begin{equation*}
\frac{1}{\hbar}\sum^{\infty}_{m = 1} \frac{(-1)^{m\bra\g, \gd\ket}}{m} \left((\log \Psi^{(m)}_{\hbar}(i t))_{12} + (\log \Psi^{(m)}_{\hbar}(i t))_{21}\right) 
\end{equation*}
as
\begin{align*}
&-\bra \g,\gd \ket   \Omega\frac{1}{\pi} \int^{\infty}_0 \arctan\left(\left(\frac{Z}{t}\right)^{-1} s\right)  \sum^{\infty}_{m = 1} e^{-m s} ds\\ &=\bra \gd,\g \ket   \Omega\frac{1}{\pi} \int^{\infty}_0 \arctan\left(\left(\frac{Z}{t}\right)^{-1} s\right) \frac{1}{e^s - 1} ds + O(\hbar).
\end{align*}
Binet's formula for the log gamma function is the identity
\begin{equation*}
\log \Gamma(z) = \left(z - \frac{1}{2}\right)\log z - z + \frac{1}{2}\log(2\pi) + \frac{1}{\pi}\int^{\infty}_0 \frac{\arctan(s/(2\pi z))}{e^{s}-1}ds
\end{equation*}
valid for $\Re(z) > 0$ (see \cite{erdelyi} p. 22 equation 9). Applying this identity shows
\begin{align*}
\frac{1}{\hbar}\sum^{\infty}_{m = 1} \frac{(-1)^{m\bra\g, \gd\ket}}{m} \log \Psi^{(m)}_{\hbar}((2\pi)^{-1}i t)_{(12)} &= \bra \gd, \g \ket \Omega \log\Lambda\left(\frac{Z(\gamma)}{t}\right)+ O(\hbar)\\
&= \log \Psi_{\gd}(t)+ O(\hbar)
\end{align*}
as required.
\end{proof}
\section{Finite uncoupled case}\label{UncoupledSec}
In this Section we spell out how to extend our results from the $A_1$ case to a finite, uncoupled variation of BPS structure. This is mostly a matter of notation.

In this case there is a finite subset $\{\g_i\} \subset \Gamma$ such that $\Omega(\pm \gamma_i)$ is nonvanishing, and we have $\bra \g_i, \g_j\ket = 0$ for all $i, j$. We also fix a reference basis $\{\beta_i\}$ for $\Gamma$. 
\begin{lem}\label{uncoupledJoyce} For a finite uncoupled variation of BPS structure we have for $k \in \Z \setminus \{0\}$
\begin{equation*}
f^{k \g_j}(Z) =  \frac{\Omega(\g_j)}{2\pi i k^2}  {\bf s}^{k\g_j}.
\end{equation*} 
(a constant, independent of $Z$) while all the other $f^{\alpha}(Z)$ vanish identically. In particular we have the symmetry $f^{\alpha} = f^{-\alpha}$. 
\end{lem}
\begin{proof} The proof is the same as that of Lemma \ref{A1joy}.
\end{proof}
Let us still denote by $(K, \nabla, C, \U, \V, g)$ the Frobenius type structure underlying a finite, uncoupled variation of BPS structure.
\begin{cor} For a finite uncoupled variation of BPS structure we have
\begin{align*}
\nabla^r &= d + \sum_{i, k \neq 0} \frac{\Omega(\g_i)}{2\pi i k^2} {\bf s}^{k \g_i} [x_{k\g_i}, - ] \frac{dZ(k\g_i)}{Z(k\g_i)},\\
\V &= \sum_{i, k \neq 0} \frac{\Omega(\g_i)}{2\pi i k^2} {\bf s}^{k \g_i} [x_{k\g_i}, - ].
\end{align*} 
\end{cor}
\begin{proof} The proof is the same as that of Corollary \ref{A1Frob}. 
\end{proof}
Just as in the $A_1$ case we write $K_{\Delta} \subset K$ for the finite dimensional subbundle spanned by the sections $x_{\alpha_i}$, where $\Delta = \{\alpha_i\} \subset \Gamma$ is a subset with $N$ elements.
\begin{thm} Let $(\Gamma, Z, \Omega)$ be a framed variation of BPS structure. For all finite $\Delta \subset \Gamma$ write $K_{\Delta} \subset K$ for the subbundle spanned by $x_{\alpha}$, $\alpha \in \Delta$, endowed with the structure induced by the canonical projection $K \to K_{\Delta}$. Then $K_{\Delta}$ is Frobenius if and only if $(\Gamma, Z, \Omega)$ is uncoupled.
\end{thm}
\begin{proof} In one direction the proof is the same as that of Lemma \ref{A1projection}. The converse is established in \cite{AnnaTomJ} Lemma 20. More precisely choose $\Delta = \{\alpha_i, \alpha_j, \alpha_k\}$ such that $\alpha_j - \alpha_k$, $\alpha_k - \alpha_i$ are active classes with $\bra \alpha_j - \alpha_k, \alpha_k - \alpha_i\ket \neq 0$ and  
\begin{equation*} 
\bra \alpha_j, \alpha_i\ket\bra \alpha_j - \alpha_k, \alpha_k - \alpha_i\ket \neq \bra \alpha_j, \alpha_k\ket \bra \alpha_k, \alpha_i\ket.
\end{equation*} 
This is always possible if $(\Gamma, Z, \Omega)$ is not uncoupled. Then it is shown in loc. cit. that the projection of $\nabla^r$ to $K_{\Delta}$ is not flat.
\end{proof}
As usual once we project to a finite-dimensional subbundle $K_{\Delta}$ we always evaluate at the geometric point $s_i = 1$, $i = 1, \cdots, N$, and we consider the meromorphic connections of $K_{\Delta}$
\begin{equation*} 
\nabla(Z) = d + \left(\frac{U(Z)}{t^2} - \frac{V}{t}\right)dt.
\end{equation*}
As in Lemma \ref{isoConnCoeff} we have
\begin{align*}
U(Z)_{ij} &= Z(\alpha_i) \delta_{ij},\\
V_{ij} &= (-1)^{\bra \alpha_i, \alpha_j\ket} \bra \alpha_i, \alpha_j\ket f^{\alpha_i - \alpha_j}.
\end{align*}
\begin{definition} Fix an active class $\g_i$ and a basis element $\beta_j$ with $\bra \g_i, \beta_j\ket \neq 0$. For all $k \geq 1$ choose $\Delta^k_{ij} = \{\alpha_{1}, \ldots \alpha_{2k}\} \subset \Gamma$ with $\{\alpha_{2m - 1}, \alpha_{2m}\} = (m (\g_i + \beta_j), m \beta_j)$. We define the (even) rank $N$ simple oscillator between $\gamma_i$, $\beta_j$ as the Frobenius bundle $K_{\Delta^{N/2}_{ij}}$.  
\end{definition}
\begin{lem} The meromorphic connections of the Frobenius bundle $K_{\Delta^{N/2}_{ij}}$ have $U$, $V$ block diagonal, with blocks
\begin{align*}
U^{(m),ij} &= \left(\begin{matrix}
m Z(\g_i + \beta_j) & 0\\ 0 & m Z(\beta_j) 
\end{matrix}\right),\,\\ 
V^{(m),ij} &= 
\frac{\bra \g_i, \beta_j\ket}{2\pi \sqrt{-1}}\Omega(\g_i)\left(\begin{matrix}
0 & (-1)^{m\bra\g_i,\beta_j\ket}\\ -(-1)^{m\bra\g_i,\beta_j\ket} & 0 
\end{matrix}\right)
\end{align*}
for $m = 1, \cdots, N/2$.  
\end{lem} 
\begin{proof} The proof is the same as that of Lemma \ref{oscillatorLem}.
\end{proof} 
\begin{definition} The rank $N$ simple oscillator of a finite, uncoupled variation of BPS structure with respect to a basis element $\beta_j$ is the Frobenius bundle
\begin{equation*}
K_{\beta_j}(N) = \bigoplus_{i | \bra \g_i, \beta_j\ket \neq 0} K_{\Delta^{N/2}_{ij}}.
\end{equation*}
In the following we denote the meromorphic connections of $K_{\beta_j}(N)$ by $\nabla_{\beta_j}$.
\end{definition}
By construction $\nabla_{\beta_j}$ splits as a direct sum
\begin{equation*}
\nabla_{\beta_j} = \bigoplus_{m, i|\bra \g_i, \beta_j\ket \neq 0} \nabla^{(m), i j} = \bigoplus_{m, i | \bra \g_i, \beta_j\ket \neq 0} d + (t^{-2} U^{(m),ij} - t^{-1}V^{(m),ij})dt.
\end{equation*} 
In particular we have the rescaling
\begin{equation*}
\bra \g_i, \beta_j\ket \mapsto \bra \g_i, \beta_j\ket \sqrt{-1}\hbar
\end{equation*}
acting on all our structures. For the meromorphic connections we have
\begin{equation*}
\nabla_{\hbar, \beta_j} = \bigoplus_{m, i | \bra \g_i, \beta_j\ket \neq 0} \nabla^{(m), i j}_{\hbar} = \bigoplus_{m, i | \bra \g_i, \beta_j\ket \neq 0} d + (t^{-2} U^{(m),ij} - t^{-1}V^{(m),ij}_{\hbar})dt.
\end{equation*}
where 
\begin{equation*}
V^{(m),ij}_{\hbar} = 
\frac{\bra \g_i, \beta_j\ket \hbar}{2\pi}\Omega(\g_i)\left(\begin{matrix}
0 & (-1)^{m\bra\g_i,\beta_j\ket}\\ -(-1)^{m\bra\g_i,\beta_j\ket} & 0 
\end{matrix}\right).
\end{equation*}
Let $Y^{(m),ij}_{\hbar}(t) = Y^{(m)}_{\hbar, pq}(t)$ be the $GL(2, \C)$ fundamental solution to $\nabla^{(m),ij}_{\hbar}$. Define 
\begin{equation} 
\Psi^{(m),ij}_{\hbar, pq}(t) = e^{Z_q/t} Y^{(m),ij}_{\hbar, pq}(t)
\end{equation}
where $Z_1 = m Z(\g_i + \beta_j)$, $Z_2 = m Z(\beta_j)$. Write $\hat{\xi} \in \C^2$, $\Pi\in \Hom(\C^2, \C)$ for the usual vector and linear functional.
\begin{thm}\label{largeNthm} We have  
\begin{align*}
&\exp\left(\frac{1}{\hbar}\sum^{\infty}_{m = 1} \sum_{i | \bra\g_i, \beta_j\ket\neq0} \frac{(-1)^{m\bra\g_i,\beta_j\ket}}{m}\Pi \log \Psi^{(m),ij}_{\hbar}((2\pi)^{-1} \sqrt{-1} t)\hat{\xi} \right) \\
&= \Psi_{\beta_j}(t) + O(\hbar).
\end{align*}
\end{thm}
\begin{proof} As in Lemma \ref{A1Stokeshbar} one proves that the Stokes rays of $\nabla^{(m), ij}_{\hbar}$ are $\pm\ell_{\g_i}$ and the corresponding Stokes factors are given by
\begin{align*}
S^{ij}_{\ell_{\g}, \hbar} &= \left(\begin{matrix} 
1 & 2 \sqrt{-1} \sinh\left(-\frac{(-1)^{m\bra\g_i,\beta_j\ket}}{2}\bra \g_i, \beta_j \ket \sqrt{-1}\hbar \Omega(\g_i)\right)\\
0 &1
\end{matrix}
\right)\\
&= \left(\begin{matrix} 
1 & (-1)^{m\bra\g_i,\beta_j\ket} \bra \g_i,\beta_j \ket \hbar \Omega(\g_i)\\
0 &1
\end{matrix}
\right) + O(\hbar^2),\\
S^{ij}_{-\ell_{\g}} &= \left(\begin{matrix} 
1 & 0\\
-2 \sqrt{-1} \sinh\left(-\frac{(-1)^{m\bra\g_i,\beta_j\ket}}{2}\bra \g_i,\beta_j \ket \sqrt{-1}\hbar \Omega(\g_i)\right) &1
\end{matrix}
\right)\\
&= \left(\begin{matrix} 
1 & 0\\
-(-1)^{m\bra\g_i,\beta_j\ket} \bra \g_i, \beta_j \ket \hbar \Omega(\g_i) &1
\end{matrix}
\right) + O(\hbar^2).
\end{align*}
As in Lemma \ref{integralEquhbar} this implies the identities
\begin{align*}
\Psi^{(m),ij}_{\hbar, 11}(t) &= 1 + O(\hbar^2),\\
\Psi^{(m),ij}_{\hbar, 12}(t) &= \frac{1}{2\pi \sqrt{-1}} (-1)^{m\bra\g_i,\beta_j\ket} \bra \g_i, \beta_j \ket \hbar \Omega(\g_j) \int_{\ell_{\g_i}} \frac{dt'}{t'}\frac{t}{t' - t} e^{-mZ(\g_i)/t'}\\& + O(\hbar^2),\\
\Psi^{(m),ij}_{\hbar, 21}(t) &= -\frac{1}{2\pi \sqrt{-1}} (-1)^{m\bra\g_i,\beta_j\ket} \bra \g_i,\beta_j \ket \hbar \Omega(\g_j)  \int_{-\ell_{\g_i}} \frac{dt'}{t'}\frac{t}{t' - t} e^{mZ(\g_i)/t'}\\& + O(\hbar^2),\\
\Psi^{(m),ij}_{\hbar, 22}(t) &= 1 + O(\hbar^2).
\end{align*}
From here we can proceed as in the proof of Proposition \ref{A1largeN}.
\end{proof}
\section{Tau functions}\label{tauSec}
Suppose $f_{\hbar}(t, Z(\g))$ is a scalar function depending on the parameter $\hbar$.
\begin{definition} A first order tau function for the scalar $\exp(f_{\hbar}(t, Z(\g)))$ is a function $\tau_{\hbar}(t, Z)$ which is invariant under common rescaling of $t$, $Z(\g)$ and satisfies the identity
\begin{equation*}
\frac{\del}{\del t} f_{\hbar} = \bra \gd,\g\ket \frac{\del}{\del Z(\g)} \log \tau_{\hbar}
\end{equation*}
in the first nontrivial term in the expansion around $\hbar = 0$.
\end{definition}
\begin{lem}\label{basicTau} The function given by
\begin{equation*}
\log \tau^{(m)}_{\hbar}(it) =  \frac{\Omega}{2\pi} \hbar \int^{\infty}_0  s\log\left(s^2 + \left(\frac{Z(\g)}{t}\right)^2\right)  e^{-m s} ds
\end{equation*}
is a first order tau function for $\exp\left(\frac{(-1)^{m\bra\g,\gd\ket}}{m}(\log \Psi^{(m)}_{\hbar}(it))_{(12)}\right)$.
\end{lem}
\begin{proof}
In the rest of the proof we write $Z = Z(\g)$ and suppress $O(\hbar^2)$ terms. According to the proof of Proposition \ref{A1largeN} we have
\begin{equation}\label{homog}
\frac{(-1)^{m\bra\g,\gd\ket}}{m}(\log \Psi^{(m)}_{\hbar}(it))_{(12)} = \bra \gd,\g \ket F\left(\frac{Z}{t}\right)
\end{equation}
where the function $F(w)$ is given by
\begin{equation*}
F(w) =  \hbar \Omega \frac{1}{\pi} \int^{\infty}_0 \arctan\left(\frac{s}{w}\right)  e^{-m s} ds.
\end{equation*}
Suppose the function $H(w)$ satisfies $H'(w) = w F'(w)$. Then we have
\begin{align*}
\frac{\del}{\del Z} H\left(\frac{Z}{t}\right) = \frac{1}{t} H'\left(\frac{Z}{t}\right) = \frac{Z}{t^2} F'\left(\frac{Z}{t}\right).
\end{align*}
From the general form \eqref{homog} we get
\begin{equation*}
\frac{(-1)^{m\bra\g,\gd\ket}}{m}\frac{\del}{\del t} (\log \Psi^{(m)}_{\hbar})_{(12)} = \bra \gd,\g \ket \frac{\del}{\del t} F\left(\frac{Z}{t}\right) = - \bra \gd, \g \ket \frac{Z}{t^2} F'\left(\frac{Z}{t}\right).
\end{equation*}
So $e^{-H}$ gives a tau function for $\exp\left(\frac{(-1)^{m\bra\g,\gd\ket}}{m}(\log \Psi^{(m)}_{\hbar}(it))_{(12)}\right)$. A solution $H(w)$ is given by choosing the primitive
\begin{equation*}
\hbar \Omega \frac{1}{\pi}\int w \frac{\del}{\del w} \arctan\left(\frac{s}{w}\right) dw = -\hbar \Omega \frac{1}{\pi} \frac{1}{2} s \log(s^2 + w^2)
\end{equation*}
and integrating in $e^{-ms}ds$.
\end{proof}
We can now prove a large rank limit in the $A_1$ case. 
\begin{cor}\label{A1largeNtau} The function
\begin{equation*}
\exp\left(\frac{1}{\hbar}\sum^{\infty}_{m = 1} \frac{(-1)^{m\bra\g,\gd\ket}}{m}  \log \Psi^{(m)}_{\hbar}((2\pi)^{-1} i t)_{(12)} \right)
\end{equation*}
(which equals $\Psi_{\gd}(t) + O(\hbar)$ by Proposition \ref{A1largeN}) admits the first order tau function
\begin{equation*}
\exp\left(\frac{1}{\hbar}\sum^{\infty}_{m = 1}  \log \tau^{(m)}_{\hbar}((2\pi)^{-1} i t) \right),
\end{equation*}
and the latter equals the tau function $\tau_{\ell}(t, Z(\g))$ of \eqref{bridTau}. In particular this implies Theorem \ref{bridThmB} in the rank $2$ case.
\end{cor}
\begin{proof} The claim that the second exponential is a first order tau function follows from Lemma \ref{basicTau} by summing over all frequencies and multiplying by $\hbar^{-1}$ throughout. To prove the second exponential equals $\tau_{\ell}(t, Z(\g))$ recall from the proof of Lemma \ref{basicTau} that
\begin{align*}
&\frac{1}{\hbar}\sum^{\infty}_{m = 1}  \log \tau^{(m)}_{\hbar}((2\pi)^{-1} i t)\\
& = -\sum^{\infty}_{m = 1} \frac{\Omega}{\pi} \int^{\infty}_0  w \frac{d}{dw} \arctan\left(\frac{s}{w}\right)|_{w = (2\pi)^{-1}\frac{Z}{t}}  e^{-m s} ds.
\end{align*}
By the proof of Proposition \ref{A1largeN} the right hand side equals
\begin{equation*}
\Omega w \frac{d}{dw} \log\Lambda(w)|_{w = \frac{Z}{t}}.
\end{equation*}
Now we use the identity 
\begin{equation*}
w \frac{d}{dw} \log\Lambda(w) = \frac{d}{dw} \log \Upsilon(w)
\end{equation*}
(see \cite{bridRH} Lemma 5.4), which follows at once from the identity for the Barnes $G$-function 
\begin{equation*}
\frac{d}{dw} \log G(w+1) = \frac{1}{2}\log(2\pi) + \frac{1}{2} - w + w\frac{d}{dw} \log \Gamma(w).
\end{equation*}
(see \cite{whit} p. 268 equation 50). The upshot is the required identity
\begin{equation*}
\frac{1}{\hbar}\sum^{\infty}_{m = 1}  \log \tau^{(m)}_{\hbar}((2\pi)^{-1} i t) = \Omega \log\Upsilon\left(\frac{Z}{t}\right) = \log \tau_{\ell}(t, Z).
\end{equation*}
The last claim that $\tau_{\ell}(t, Z)$ is a tau function for $\Psi(t)$ now follows from the fact that both functions are independent of $Z(\gd)$.
\end{proof}
We consider now the case of a finite, uncoupled variation of BPS structure, and follow the notation of Section \ref{UncoupledSec}. In particular we have a basis $\{\beta_j\}$, yielding local coordinates $Z(\beta_j)$. The active classes are $\{\g_i\}$, and we write
\begin{equation*}
\g_i = \sum_p c_{ip} \beta_p.
\end{equation*}
Recall we have elementary simple oscillators $\nabla^{(m), ij}_{\hbar}$, or equivalently in terms of solutions the functions $\Psi^{(m),ij}_{\hbar}$. 
\begin{lem} Fix $i, j$. The function given by
\begin{equation*}
\log \tau^{(m), i}_{\hbar}(\sqrt{-1}t) =   \frac{\Omega(\g_i)}{2\pi}\hbar \int^{\infty}_0  s\log\left(s^2 + \left(\frac{Z(\g_i)}{t}\right)^2\right)  e^{-m s} ds
\end{equation*}
is a first order tau function for the scalar 
\begin{equation*}
\exp\left(\frac{(-1)^{m\bra\g_i,\beta_j\ket}}{m}(\log \Psi^{(m), ij}_{\hbar}(\sqrt{-1}t))_{(12)}\right).
\end{equation*}
\end{lem}
\begin{proof} The proof is the same as that of Lemma \ref{basicTau}.
\end{proof} 
Suppose $v^{j}_{\hbar}(t, Z)$ is a vector function of the local coordinates $Z(\beta_k)$, with one component for each $\beta_j$, depending on the additional parameter $\hbar$.
\begin{definition} A first order tau function for the vector $\exp(v^j_{\hbar}(t, Z))$ is a scalar function $\tau_{\hbar}(t, Z)$ which is invariant under common rescaling of $t$, $Z$ and satisfies  
\begin{equation*}
\frac{\del}{\del t} v^{j}_{\hbar} = \sum_p \bra \beta_j, \beta_p\ket \frac{\del}{\del Z(\beta_p)}\log\tau_{\hbar}
\end{equation*}
for all $j$.
\end{definition}
\begin{thm} Fix a finite, uncoupled variation of BPS structure as above. The vector function
\begin{equation*}
\exp\left(\frac{1}{\hbar}\sum^{\infty}_{m = 1} \sum_{i | \bra\g_i, \beta_j\ket\neq 0}\frac{(-1)^{m\bra\g_i,\beta_j\ket}}{m} \log \Psi^{(m),ij}_{\hbar}((2\pi)^{-1} \sqrt{-1} t)_{(12)} \right)
\end{equation*}
(which equals the vector $\Psi_{\beta_j}(t) + O(\hbar)$ by Theorem \ref{largeNthm}) admits the first order tau function
\begin{equation}\label{generalTau}
\exp\left(\frac{1}{\hbar}\sum^{\infty}_{m = 1} \sum_i \log \tau^{(m), i}_{\hbar}((2\pi)^{-1} i t)_{(12)} \right),
\end{equation}
and the latter equals $\tau_{\ell}(t, Z)$. This implies in particular Theorem \ref{bridThmB}.
\end{thm}
\begin{proof}
Fix $i, j$, and evaluate at $(2\pi)^{-1}t$. We have
\begin{align*}
\frac{\del}{\del t} (\log \Psi^{(m), ij}_{\hbar})_{(12)} &= \bra \g_i, \beta_j\ket\frac{\del}{\del Z(\g_i)} \log \tau^{(m), i}_{\hbar}\\
&=\sum_p \bra \beta_p, \beta_j\ket c_{ip} \frac{\del}{\del Z(\g_i)} \log \tau^{(m), i}_{\hbar}\\
&= \sum_p \bra \beta_p, \beta_j\ket \frac{\del Z(\g_i)}{\del Z(\beta_p)} \frac{\del}{\del Z(\g_i)} \log \tau^{(m), i}_{\hbar}\\
&= \sum_p \bra \beta_p, \beta_j\ket \frac{\del}{\del Z(\beta_p)} \log \tau^{(m), i}_{\hbar}.
\end{align*}
To prove the first part of the claim sum over all $i$ and note that the right hand side vanishes when $\bra \g_i, \beta_j\ket = 0$. Arguing as in Corollary \ref{A1largeNtau} shows that the function \ref{generalTau} equals $\tau_{\ell}(t, Z) = \prod_i \Upsilon^{\Omega(\g_i)}\left(\frac{Z(\g_i)}{t}\right)$ as required. The last claim that $\tau_{\ell}(t, Z)$ is a tau function for $\Psi(t)$ follows at once.
\end{proof}

\end{document}